\documentclass[a4paper,11pt]{article}

\usepackage[utf8x]{inputenc}
\usepackage[T1]{fontenc}
\usepackage{lmodern}

\usepackage[colorlinks]{hyperref}
\hypersetup{linkcolor=blue, citecolor=red}

\newcommand{\Aut}{{\rm Aut}}

\usepackage{bbm}
\usepackage[bbgreekl]{mathbbol}
\usepackage{nccrules}
\usepackage[nottoc]{tocbibind}
\usepackage[intlimits,leqno]{amsmath}
\usepackage{amsthm}
\usepackage{amssymb}
\usepackage{amsmath}
\usepackage{mathrsfs}
\usepackage{stmaryrd}
\usepackage{latexsym,bm}
\usepackage[capitalise]{cleveref}
\usepackage{xspace}
\usepackage[all]{xy}

\usepackage{paralist} \setdefaultitem{$\star$}{}{}{}

\usepackage{geometry}
\newcommand{\GL}{{\rm GL}}
\newcommand{\SL}{{\rm SL}}
\newcommand{\Sh}{{\rm Sh}}

\newcommand{\trdeg}{\rm tr.deg}

\newcommand{\Gal}{{\rm Gal}}

\newcommand{\MT}{{\rm MT}}
\newcommand{\der}{{\rm der}}

\newcommand{\ad}{{\rm ad}}

\newcommand{\CC}{{\mathbb C}}
\newcommand{\RR}{{\mathbb R}}

\newcommand{\ZZ}{{\mathbb Z}}
\newcommand{\QQ}{{\mathbb Q}}
\newcommand{\NN}{{\mathbb N}}

\newcommand{\HH}{{\mathbb H}}

\newcommand{\SSS}{{\mathbb S}}
\newcommand{\AAA}{{\mathbb A}}

\newcommand{\Opt}{{\rm Opt}}

\theoremstyle{plain}
\newtheorem{proposition}{Proposition}[section]
\newtheorem{conjecture}[proposition]{Conjecture}
\newtheorem{lemma}[proposition]{Lemma}
\newtheorem{theorem}[proposition]{Theorem}
\newtheorem{corollary}[proposition]{Corollary}

\theoremstyle{definition}
\newtheorem{definition}[proposition]{Definition}
\newtheorem{definition-theorem}[proposition]{Definition-Theorem}
\newtheorem{definition-proposition}[proposition]{Definition-Proposition}

\theoremstyle{remark}
\newtheorem{remark}[proposition]{Remark}

\numberwithin{equation}{proposition}

\renewcommand{\Im}{\ensuremath{\mathrm{Im}\xspace}}

\title{Applications of the hyperbolic Ax-Schanuel conjecture}
\author{Christopher Daw \and Jinbo Ren}
\date{}

\begin{document}

\maketitle

\begin{abstract}

In 2014, Pila and Tsimerman gave a proof of the Ax-Schanuel conjecture for the $j$-function and, with Mok, have recently announced a proof of its generalization to any (pure) Shimura variety. We refer to this generalization as the hyperbolic Ax-Schanuel conjecture. In this article, we show that the hyperbolic Ax-Schanuel conjecture can be used to reduce the Zilber-Pink conjecture for Shimura varieties to a problem of point counting. We further show that this point counting problem can be tackled in a number of cases using the Pila-Wilkie counting theorem and several arithmetic conjectures. Our methods are inspired by previous applications of the Pila-Zannier method and, in particular, the recent proof by Habegger and Pila of the Zilber-Pink conjecture for curves in abelian varieties.

\end{abstract}

\begin{flushleft}
\end{flushleft}

\tableofcontents

\section{Introduction}\label{sec:intro}

The Ax-Schanuel theorem \cite{Ax71} is a result regarding the transcendence degrees of fields generated over the complex numbers by power series and their exponentials. Formulated geometrically for the uniformization maps of algebraic tori, it has inspired analogous statements for the uniformization maps of abelian varieties and Shimura varieties. The former, following from another theorem of Ax \cite{Ax72}, has recently been used by Habegger and Pila in their proof of the Zilber-Pink conjecture for curves in abelian varieties \cite{hp:o-min}. 

Habegger and Pila also extended the Pila-Zannier strategy to the Zilber-Pink conjecture for products of modular curves. Their method relies on an Ax-Schanuel conjecture for the $j$-function and is conditional on their so-called large Galois orbits conjecture. The purpose of this paper is to show that the Pila-Zannier strategy can be extended to the Zilber-Pink conjecture for general Shimura varieties.

This conjecture can just as easily be stated in the generality of {\bf mixed} Shimura varieties but, in this article, we will restrict our attention to {\bf pure} Shimura varieties, though we have no explicit reason to believe that the methods presented here will not extend to the mixed setting. We begin by stating a conjecture of Pink. We note that, throughout this article, unless preceded by the word {\bf Shimura}, varieties (and, indeed, subvarieties) will be assumed {\bf geometrically irreducible}.

\begin{conjecture}[cf. \cite{pink:generalisation}, Conjecture 1.3]\label{zp}
Let $\Sh_K(G,\mathfrak{X})$ be a Shimura variety and, for any integer $d$, let $\Sh_K(G,\mathfrak{X})^{[d]}$ denote the union of the special subvarieties of $\Sh_K(G,\mathfrak{X})$ having codimension at least $d$. Let $V$ be a {\bf Hodge generic} subvariety of $\Sh_K(G,\mathfrak{X})$. Then 
\begin{align*}
V\cap \Sh_K(G,\mathfrak{X})^{[1+\dim V]}
\end{align*}
is not Zariski dense in $V$.
\end{conjecture}

The heuristics of this conjecture are as follows. For two subvarieties $V$ and $W$ of $\Sh_K(G,\mathfrak{X})$, such that the codimension of $W$ is at least $1+\dim V$, we expect $V\cap W=\emptyset$. Even if we fix $V$ and take the union of $V\cap W$ for countably many $W$ of codimension at least $1+\dim V$, the resulting set should still be rather small in $V$ unless, of course, $V$ was not sufficiently generic in $\Sh_K(G,\mathfrak{X})$. Pink's conjecture turns this expectation into an explicit statement about the intersection of Hodge generic subvarieties with the special subvarieties of small dimension.

Conjecture \ref{zp} can also be formulated for algebraic tori, abelian varieties, or even semi-abelian varieties, though Conjecture \ref{zp} for mixed Shimura varieties implies all of these formulations (see \cite{pink:generalisation}). When $V$ is a curve, defined over a number field, and contained in an algebraic torus, we obtain a theorem of Maurin \cite{maurin}. We also note that Capuano, Masser, Pila, and Zannier have recently applied the Pila-Zannier method in this setting \cite{CMPZ:torus}. When $V$ is a curve, defined over a number field, and contained in an abelian variety, we obtain the recent theorem of Habegger and Pila \cite{hp:o-min}, and it is the ideas presented there that form the basis for this article. Habegger and Pila had given some partial results when $V$ is a curve, defined over a number field, and contained in the Shimura variety $\CC^n$ \cite{hp:beyond}, and Orr has recently generalized their results to a curve contained in $\mathcal{A}^2_g$ (see \cite{orr:unlikely} for more details).

We should point out that Conjecture \ref{zp} implies the Andr\'{e}-Oort conjecture.

\begin{conjecture}[Andr\'{e}-Oort]
Let $\Sh_K(G,\mathfrak{X})$ be a Shimura variety and let $V$ be a subvariety of $\Sh_K(G,\mathfrak{X})$ such that the special points of $\Sh_K(G,\mathfrak{X})$ in $V$ are Zariski dense in $V$. Then $V$ is a special subvariety of $\Sh_K(G,\mathfrak{X})$.
\end{conjecture}
To see this, we may assume that $V$ is Hodge generic in $\Sh_K(G,\mathfrak{X})$. Then, since special points have codimension $\dim \Sh_K(G,\mathfrak{X})$, Conjecture \ref{zp} implies that, either $\dim V=\dim \Sh_K(G,\mathfrak{X})$, in which case $V$ is a connected component of $\Sh_K(G,\mathfrak{X})$ and, in particular, a special subvariety of $\Sh_K(G,\mathfrak{X})$, or the set of special points of $\Sh_K(G,\mathfrak{X})$ in $V$ are not Zariski dense in $V$. 

In precisely the same fashion, the Zilber-Pink conjecture for abelian varieties implies the Manin-Mumford conjecture.

The Andr\'e-Oort conjecture has a rich history of its own. Here, we simply recall that it was recently settled for $\mathcal{A}_g$ by Pila and Tsimerman \cite{pt:axlindemann,tsimerman:AO}, thanks to recent progress on the Colmez conjecture due to Andreatta, Goren, Howard, Madapusi Pera \cite{AGHM:colmez} and Yuan and Zhang \cite{YZ:colmez}, and it is known to hold for all Shimura varieties under conjectural lower bounds for Galois orbits of special points due to the work of Orr, Klingler, Ulmo, Yafaev, and the first author \cite{Daw2016,kuy:ax-lindemann,uy:galois}. Furthermore, Gao has generalized these proofs to all mixed Shimura varieties \cite{gao:AO,gao:reduction}.

\medskip

In his work on Schanuel's conjecture, Zilber made his own conjecture on unlikely intersections \cite{zilber:exponential}, which was closely related to the independent work of Bombieri, Masser, and Zannier \cite{bombieri2007anomalous}. To describe Zilber's formulation, we require the following definition.

\begin{definition}
Let $\Sh_K(G,\mathfrak{X})$ be a Shimura variety and let $V$ be a subvariety of $\Sh_K(G,\mathfrak{X})$. A subvariety $W$ of $V$ is called {\bf atypical} with respect to $V$ if there is a special subvariety $T$ of $\Sh_K(G,\mathfrak{X})$ such that $W$ is an irreducible component of $V\cap T$ and
\begin{align*}
\dim W>\dim V+\dim T-\dim \Sh_K(G,\mathfrak{X}).
\end{align*}
We denote by ${\rm Atyp}(V)$ the union of the subvarieties of $V$ that are atypical with respect to $V$.
\end{definition}

Zilber's conjecture, formulated for Shimura varieties, is then as follows. 

\begin{conjecture}[cf. \cite{hp:o-min}, Conjecture 2.2]\label{zp'}
Let $\Sh_K(G,\mathfrak{X})$ be a Shimura variety and let $V$ be a subvariety of $\Sh_K(G,\mathfrak{X})$. Then ${\rm Atyp}(V)$ is equal to a finite union of atypical subvarieties of $V$.
\end{conjecture}

Since there are only countably many special subvarieties of $\Sh_K(G,\mathfrak{X})$, the conjecture is equivalent to the statement that $V$ contains only finitely many subvarieties that are atypical with respect to $V$ and maximal with respect to this property.

We will see that Conjecture \ref{zp'} strengthens Conjecture \ref{zp} and, therefore, it is Conjecture \ref{zp'} that we refer to as the {\bf Zilber-Pink} conjecture. Habegger and Pila obtained a proof of the Zilber-Pink conjecture for products of modular curves assuming the weak complex Ax conjecture and the large Galois orbits conjecture. Subsequently, Pila and Tsimerman obtained the weak complex Ax conjecture as a corollary to their proof of the Ax-Schanuel conjecture for the $j$-function \cite{pila2014ax}. Habegger and Pila had previously verified the large Galois orbits conjecture for so-called asymmetric curves \cite{hp:beyond}.

This article seeks to generalize the ideas of \cite{hp:o-min} to general Shimura varieties. Hence, we will have to make generalizations of the previously mentioned hypotheses. The foremost of which will be the statement from functional transcendence, namely, the hyperbolic Ax-Schanuel conjecture that generalizes the Ax-Schanuel conjecture for the $j$-function to general Shimura varieties. Our main result (Theorem \ref{main theorem}) is that, under the hyperbolic Ax-Schanuel conjecture, the Zilber-Pink conjecture can be reduced to a problem of point counting. However, given that Mok, Pila, and Tsimerman have recently announced a proof of the hyperbolic Ax-Schanuel conjecture \cite{MPT:AS}, this result is now very likely unconditional. Besides the hyperbolic Ax-Schanuel conjecture, our main input will be the theory of o-minimality and, in particular, the fact that the uniformization map of a Shimura variety is definable in $\RR_{\rm an,exp}$ when it is restricted to an appropriate fundamental domain.

After establishing the main result, we attempt to tackle the point counting problem using the now famous Pila-Wilkie counting theorem. To do so, we formulate several arithmetic conjectures that are inspired by previous applications of the Pila-Zannier strategy. In this vein, our paper is very much in the spirit of \cite{ullmo:applications}, which, at the time, reduced the Andr\'e-Oort conjecture to a point counting problem and then explained how various conjectural ingredients, namely, the hyperbolic Ax-Lindemann conjecture, lower bounds for Galois orbits of special points, upper bounds for the heights of pre-special points, and the definability of the uniformization map, could be combined to deliver a proof of the Andr\'e-Oort conjecture.

Our arithmetic hypotheses are (1) lower bounds for Galois orbits of so-called optimal points (see Definition \ref{defopt}), which we also refer to as the large Galois orbits conjecture, and (2) upper bounds for the heights of pre-special subvarieties. Hypothesis (1) generalizes the (in some cases still conjectural) lower bounds for Galois orbits of special points (when such special points are also maximal special subvarieties), and also generalizes the large Galois orbits conjecture of Habegger and Pila. Hypothesis (2) generalizes the upper bounds for heights of pre-special points, which were proved by Orr and the first author \cite{Daw2016}. However, we also show that it is possible to replace hypothesis (2) with two other arithmetic hypotheses, namely, (3) upper bounds for the degrees of fields associated with special subvarieties, and (4) upper bounds for the heights of lattice elements. Hypothesis (3) is a replacement for the fact that, for an abelian variety, its abelian subvarieties can be defined over a fixed finite extension of the base field. Hypothesis (4) is an analogue of a known result for abelian varieties. We verify hypotheses (2), (3), and (4) for a product of modular curves.

\paragraph{Acknowledgements}
The first author would like to thank the EPSRC, as well as Jonathan Pila, for the opportunity to be part of the project Model Theory, Functional Transcendence, and Diophantine Geometry as a postdoctoral research assistant. He would like to thank Linacre College, Oxford, the Mathematical Insitute at the University of Oxford, and the Department of Mathematics and Statistics at the University of Reading, all for providing excellent working conditions. Finally, he would like to thank Martin Orr, Jonathan Pila, Harry Schmidt, Emmanuel Ullmo, and Andrei Yafaev for several valuable discussions. The second author is grateful to the Institut des Hautes Études Scientifiques and the Université Paris Saclay for providing great environments in which to work. He would like to thank his supervisor Emmanuel Ullmo for regular discussions and constant support during the preparation of this article and he would like to thank Mikhail Borovoi, Philipp Habegger, Ziyang Gao, Martin Orr, and Jonathan Pila for several useful discussions. His work was supported by grants from Région l'Île de France. Both authors would like to thank Martin Orr, for sharing drafts of his preprint \cite{orr:unlikely}. They would also like to thank Bruno Klingler, as well as the anonymous referee, for their many detailed comments.

\paragraph{Conventions}
\begin{itemize}
  \item Throughout this paper, {\bf definable} means definable in the o-minimal structure $\RR_{\textup{an,exp}}$.
  \item Unless preceded by the word {\bf Shimura}, varieties (and, indeed, subvarieties) will be assumed {\bf geometrically irreducible}.
  \item By a subvariety, we will always mean a {\bf closed} subvariety.
\end{itemize}
\paragraph{Index of notations}
We collect here the main symbols appearing in this article.
\begin{itemize}
  \item $\langle  W\rangle$ is the smallest special subvariety containing $W$.
  \item $\langle W \rangle_{\rm ws}$ is the smallest weakly special subvariety containing $W$.
  \item $\langle A \rangle_{\rm Zar}$ is the smallest algebraic subvariety containing $A$.
  \item $\langle A \rangle_{\rm geo}$ is the smallest totally geodesic subvariety containing $A$.
  \item $\delta(W):=\dim \langle W \rangle-\dim W$
  \item $\delta_{\rm ws}(W):=\dim \langle W \rangle_{\rm ws}-\dim W$
  \item $\delta_{\rm Zar}(A):=\dim \langle A \rangle_{\rm Zar}-\dim A$
  \item $\delta_{\rm geo}(A):=\dim \langle A \rangle_{\rm geo}-\dim A$
  \item $\Opt(V)$ is the set of subvarieties of $V$ that are optimal in $V$.
  \item $\Opt_0(V)$ is the set of points of $V$ that are optimal in $V$.
  \item $G^{\ad}$ is the adjoint group of $G$ i.e. the quotient of $G$ by its centre.
  \item $G^{\rm der}$ is derived group of $G$.
  \item $G^{\circ}$ is the (Zariski) connected component of $G$ containing the identity.
  \item $G_H:=H Z_G(H)^{\circ}$ whenever $H$ is a subgroup of $G$.
  \item $G(\RR)^+$ is the (archimedean) connected component of $G(\RR)$ containing the identity.
\end{itemize}

\section{Special and weakly special subvarieties}\label{preli}
Let $( G,\mathfrak{X})$ be a Shimura datum and let $K$ be a compact open subgroup of ${ G}(\AAA_f)$, where $\AAA_f$ will henceforth denote the finite rational ad\`eles. Let $\Sh_K({ G},\mathfrak{ X})$ denote the corresponding {\bf Shimura variety}. By this, we mean the complex quasi-projective algebraic variety such that $\Sh_K({ G},\mathfrak{ X})(\CC)$ is equal to the image of
\begin{align}\label{double}
{ G}(\QQ)\backslash[\mathfrak{ X}\times({ G}(\AAA_f)/K)]
\end{align}
under the canonical embedding into complex projective space given by Baily and Borel \cite{bb:compactification}. We will identify (\ref{double}) with $\Sh_K({ G},\mathfrak{ X})(\CC)$. We recall that, on $\mathfrak{ X}\times({ G}(\AAA_f)/K)$, the action of ${ G}(\QQ)$ is the diagonal one. 

Let $X$ be a connected component of $\mathfrak{X}$ and let ${ G}(\QQ)_+$ be the subgroup of ${ G}(\QQ)$ acting on it. For any $g\in G(\AAA_f)$, we obtain a {\bf congruence} subgroup $\Gamma_g$ of ${ G}(\QQ)_+$ by intersecting it with $gKg^{-1}$. Furthermore, the locally symmetric variety $\Gamma_g\backslash X$ is contained in (\ref{double}) via the map that sends the class of $x$ to the class of $(x,g)$.
If we take the disjoint union of the $\Gamma_g\backslash X$ over a (finite) set of representatives for
\begin{align*}
{ G}(\QQ)_+\backslash{ G}(\AAA_f)/K,
\end{align*}
the corresponding inclusion map is a bijection. 

\begin{definition}
For any compact open subgroup $K'$ of $G(\AAA_f)$ contained in $K$, we obtain a finite morphism
\begin{align*}
\Sh_{K'}(G,\mathfrak{X})\rightarrow\Sh_K(G,\mathfrak{X}),
\end{align*}
given by the natural projection. Furthermore, for any $a\in G(\AAA_f)$, we obtain an isomorphism
\begin{align*}
\Sh_K(G,\mathfrak{X})\rightarrow\Sh_{a^{-1}Ka}(G,\mathfrak{X})
\end{align*}
sending the class of $(x,g)$ to the class of $(x,ga)$. We let $T_{K,a}$ denote the map on algebraic cycles of $\Sh_K(G,\mathfrak{X})$ given by the algebraic correspondence 
\begin{align*}
\Sh_K(G,\mathfrak{X})\leftarrow\Sh_{K\cap aKa^{-1}}(G,\mathfrak{X})\rightarrow\Sh_{a^{-1}Ka\cap K}(G,\mathfrak{X})\rightarrow\Sh_K(G,\mathfrak{X}),
\end{align*}
where the outer arrows are the natural projections and the middle arrow is the isomorphism given by $a$. We refer to a map of this sort as a {\bf Hecke correspondence}.
\end{definition}

\begin{definition}
Let $({ H},{ \mathfrak{X}_H})$ be a Shimura subdatum of $({ G},\mathfrak{ X})$ and let $K_{ H}$ denote a compact open subgroup of ${ H}(\AAA_f)$ contained in $K$. The natural map
\begin{align*}
{ H}(\QQ)\backslash[\mathfrak{ X}_H\times({ H}(\AAA_f)/K_H)]\rightarrow{ G}(\QQ)\backslash[\mathfrak{ X}\times({ G}(\AAA_f)/K)]
\end{align*}
 yields a finite morphism of Shimura varieties
\begin{align*}
\Sh_{K_{ H}}({ H}, \mathfrak{X}_H)\rightarrow\Sh_K({ G},\mathfrak{ X})
\end{align*}
(see, for example, \cite{pink:published}, Facts 2.6), and we refer to the image of any such morphism as a {\bf Shimura subvariety} of $\Sh_K({ G},\mathfrak{ X})$. 

For any Shimura subvariety $Z$ of $\Sh_K({ G},\mathfrak{ X})$ and any $a\in G(\AAA_f)$, we refer to any irreducible component of $T_{K,a}(Z)$ as a {\bf special subvariety} of $\Sh_K({ G},\mathfrak{ X})$.
\end{definition}

Recall that, by definition, $\mathfrak{ X}$ is a ${ G}(\RR)$ conjugacy class of morphisms from $\mathbb{S}$ to ${ G}_{\RR}$ and the {\bf Mumford-Tate group} $\MT(x)$ of $x\in\mathfrak{ X}$ is defined as the smallest $\QQ$-subgroup $ H$ of $ G$ such that $x$ factors through ${ H}_{\RR}$. If we let $\mathfrak{ X}_M$ denote the ${M}(\RR)$ conjugacy class of $x\in X$, where $M:=\MT(x)$, then $({ M},\mathfrak{ X}_{ M})$ is a Shimura subdatum of $({ G},{ \mathfrak{X}})$. In particular, if we let $X_{ M}$ denote a connected component of $\mathfrak{X}_M$ contained in $X$, then the image of $X_M$ in $\Gamma_g\backslash X$, for any $g\in G(\AAA_f)$, is a special subvariety of $\Sh_K({ G},\mathfrak{ X})$, and it is easy to see that every special subvariety of $\Sh_K({ G},\mathfrak{ X})$ arises this way.

Of course, if $x\in X_M$, then $X_M$ is equal to the $M(\RR)^+$ conjugacy class of $x$. Furthermore, the action of $M(\RR)$ on $\mathfrak{X}_M$ factors through $M^{\ad}(\RR)$ and the group $M^{\ad}$ is equal to the direct product of its $\QQ$-simple factors. Therefore, we can write $M^{\ad}$ as a product
\begin{align*}
M^{\ad}=M_1\times M_2
\end{align*} 
of two normal $\QQ$-subgroups, either of which may (by choice or necessity) be trivial, and we thus obtain a corresponding splitting
\begin{align*}
X_M=X_1\times X_2.
\end{align*}
For any such splitting, and any $x_1\in X_1$ or $x_2\in X_2$, we refer to the image of $\{x_1\}\times X_2$ or $X_1\times\{x_2\}$ in $\Gamma_g\backslash X$, for any $g\in G(\AAA_f)$, as a {\bf weakly special subvariety} of $\Sh_K({ G},\mathfrak{ X})$. In particular, every special subvariety of $\Sh_K({ G},\mathfrak{ X})$ is a weakly special subvariety of $\Sh_K({ G},\mathfrak{ X})$. By \cite{Moonen:linear1}, Section 4, the weakly special subvarieties of $\Sh_K({ G},\mathfrak{ X})$ are precisely those subvarieties of $\Sh_K({ G},\mathfrak{ X})$ that are totally geodesic in $\Sh_K({ G},\mathfrak{ X})$. Furthermore, a weakly special subvariety of $\Sh_K({ G},\mathfrak{ X})$ is a special subvariety of $\Sh_K({ G},\mathfrak{ X})$ if and only if it contains a special subvariety of dimension zero, henceforth known as a {\bf special point}.

\begin{remark}\label{introreductions}
The following observations will facilitate various reductions.
\begin{itemize}
\item Let $K'$ be a compact open subgroup of $G(\AAA_f)$ contained in $K$. By definition, a subvariety $Z$ of $\Sh_K({ G},\mathfrak{ X})$ is a (weakly) special subvariety of $\Sh_K({ G},\mathfrak{ X})$ if and only if any irreducible component of the inverse image of $Z$ in $\Sh_{K'}({ G},\mathfrak{ X})$ is a (weakly) special subvariety of $\Sh_{K'}({ G},\mathfrak{ X})$.

\item For any $a\in G(\AAA_f)$, a subvariety $Z$ of $\Sh_K({ G},\mathfrak{ X})$ is a (weakly) special subvariety of $\Sh_K({ G},\mathfrak{ X})$ if and only if any irreducible component of $T_{K,a}(Z)$ is a (weakly) special subvariety of $\Sh_{K'}({ G},\mathfrak{ X})$.

\item If we let $G^{\ad}$ denote the adjoint group of $G$ i.e. the quotient of $G$ by its centre, we obtain another Shimura datum $(G^{\ad},\mathfrak{X}^{\ad})$, known as the {\bf adjoint Shimura datum} associated with $(G,\mathfrak{X})$. For any compact open subgroup $K^{\ad}$ of $G^{\ad}(\AAA_f)$ containing the image of $K$, we obtain a finite morphism
\begin{align*}
\Sh_K(G,\mathfrak{X})\rightarrow\Sh_{K^{\ad}}(G^{\ad},\mathfrak{X}^{\ad}).
\end{align*}
As in \cite{EY:subvar}, Proposition 2.2, a subvariety $Z$ of $\Sh_{K^{\ad}}(G^{\ad},\mathfrak{X}^{\ad})$ is a special subvariety of $\Sh_{K^{\ad}}(G^{\ad},\mathfrak{X}^{\ad})$ if and only if any irreducible component of its inverse image in $\Sh_K(G,\mathfrak{X})$ is a special subvariety of $\Sh_K(G,\mathfrak{X})$.
\end{itemize}
\end{remark}

By \cite{pink:published}, Remark 4.9, for any subvariety $W$ of $\Sh_K({ G},\mathfrak{ X})$, there exists a {\bf smallest} weakly special subvariety $\langle W\rangle_{\rm ws}$ of $\Sh_K({ G},\mathfrak{ X})$ containing $W$ and a {\bf smallest} special subvariety $\langle W\rangle$ of $\Sh_K({ G},\mathfrak{ X})$ containing $W$. We note that here, and throughout, our notations and terminology regarding subvarieties often differ from those found in \cite{hp:o-min}.

\section{The Zilber-Pink conjecture}\label{zilber-pink}

For the remainder of this article, we fix a Shimura datum $(G,\mathfrak{X})$ and we let $X$ be a connected component of $\mathfrak{X}$. We fix a compact open subgroup $K$ of $G(\AAA_f)$ and we let
\begin{align*}
\Gamma:=G(\QQ)_+\cap K,
\end{align*}		
where $G(\QQ)_+$ is the subgroup of $G(\QQ)$ acting on $X$. We denote by $S$ the connected component $\Gamma\backslash X$ of $\Sh_K(G,\mathfrak{X})$.

As in \cite{hp:o-min}, we will consider an equivalent formulation of Conjecture \ref{zp'} using the language of optimal subvarieties. 

\begin{definition}
Let $W$ be a subvariety of $S$. We define the {\bf defect} of $W$ to be
\begin{align*}
\delta(W):=\dim\left<W\right>-\dim W.
\end{align*}
\end{definition}

\begin{definition}\label{defopt}
Let $V$ be a subvariety of $S$ and let $W$ be a subvariety of $V$. Then $W$ is called {\bf optimal} in $V$ if, for any subvariety $Y$ of $S$, 
\begin{align*}
W\subsetneq Y\subseteq V\implies\delta(Y)>\delta(W).
\end{align*}		
We denote by $\textup{Opt}(V)$ the set of all subvarieties of $V$ that are optimal in $V$.
\end{definition}

\begin{remark}
Let $V$ be a subvariety of $S$. First note that $V\in{\rm Opt}(V)$. Secondly, if $W\in{\rm Opt}(V)$, then $W$ is an irreducible component of
\begin{align*}
\langle W\rangle\cap V.
\end{align*}
\end{remark}

\begin{conjecture}[cf. \cite{hp:o-min}, Conjecture 2.6]\label{zp2}
Let $V$ be a subvariety of $S$. Then $\textup{Opt}(V)$ is finite.
\end{conjecture}

Observe that a maximal special subvariety of $V$ is an optimal subvariety of $V$. Therefore, Conjecture \ref{zp2} immediately implies that $V$ contains only finitely many maximal special subvarieties, which is another formulation of the Andr\'{e}-Oort conjecture for $V$.

\begin{lemma}\label{equivalence}
The Zilber-Pink conjecture (Conjecture \ref{zp'}) is equivalent to Conjecture \ref{zp2}.
\end{lemma}

\begin{proof}
Consider the situation described in the statement of Conjecture \ref{zp'}. By Remark \ref{introreductions}, we suffer no loss in generality if we assume that $V$ is contained in $S$. Then the result follows from \cite{hp:o-min}, Lemma 2.7.
\end{proof}

\begin{lemma}
The Zilber-Pink conjecture implies Conjecture \ref{zp}.
\end{lemma}

\begin{proof}
By Lemma \ref{equivalence}, it suffices to show that Conjecture \ref{zp2} implies Conjecture \ref{zp}. 

Consider the situation described in Conjecture \ref{zp}. By Remark \ref{introreductions}, we suffer no loss in generality if we assume that $V$ is contained in $S$. Let $P$ be a point belonging to 
\begin{align*}
V\cap\Sh_K(G,\mathfrak{X})^{[1+\dim V]}.
\end{align*}
Let $W$ be a subvariety of $V$ that is optimal in $V$ and contains $P$ such that 
\begin{align*}
\delta(W)\leq\delta(P)=\dim \langle P\rangle.
\end{align*}
Since $P$ belongs to a special subvariety of codimension at least $\dim V+1$ and $V$ is Hodge generic in $\Sh_K(G,\mathfrak{X})$, we have 
\begin{align*}
\dim \langle P\rangle \leq \dim S-\dim V-1=\dim \langle V\rangle -\dim V-1<\delta(V).
\end{align*}
Therefore, $\delta(W)<\delta(V)$ and we conclude that $W$ is not $V$. According to Conjecture \ref{zp2}, the union of the subvarieties belonging to ${\rm Opt}(V)\setminus V$ is not Zariski dense in $V$.
\end{proof}

\section{The defect condition}\label{defect}

In this section, we prove Habegger and Pila's defect condition (Proposition \ref{dc}) for Shimura varieties, and thus show that a subvariety that is optimal is weakly optimal.

\begin{definition}
Let $W$ be a subvariety of $S$. We define the {\bf weakly special defect} of $W$ to be
\begin{align*}
\delta_{\rm ws}(W):=\dim\langle W\rangle_{\rm ws}-\dim W.
\end{align*}
We note that, in \cite{hp:o-min}, this notion was referred to as geodesic defect.
\end{definition}

\begin{definition}
If $V$ is a subvariety of $S$ and $W$ a subvariety of $V$, then $W$ is called {\bf weakly optimal} in $V$ if, for any subvariety $Y$ of $S$, 
\begin{align*}
W\subsetneq Y\subseteq V\implies\delta_{\rm ws}(Y)>\delta_{\rm ws}(W).
\end{align*}		
\end{definition}

\begin{remark}
Let $V$ be a subvariety of $S$ and $W$ a subvariety of $V$. If $W$ is weakly optimal in $V$, then $W$ is an irreducible component of
\begin{align*}
\langle W\rangle_{\rm ws}\cap V.
\end{align*}
\end{remark}

\begin{proposition}[cf. \cite{hp:o-min}, Proposition 4.3]\label{dc}
The following {\bf defect condition} holds.

Let $W\subseteq Y$ be two subvarieties of $S$. Then
\begin{align*}
\delta(Y)-\delta_{\rm ws}(Y)\leq \delta(W)-\delta_{\rm ws}(W).
\end{align*}
\end{proposition}

\begin{proof}
We need to show that
\begin{align*}
\dim\langle Y\rangle-\dim\langle Y\rangle_{\rm ws}\leq\dim\langle W\rangle-\dim\langle W\rangle_{\rm ws}.
\end{align*}
By Remark \ref{introreductions}, we can and do assume that $G$ is the generic Mumford-Tate group on $X$, that it is equal to $G^{\ad}$, and that $Y$ is Hodge generic in $S$. By definition, there exists a decomposition
\begin{align*}
{ G}={ G}_1\times{ G}_2,
\end{align*}
which induces a splitting
\begin{align*}
X=X_1\times X_2,
\end{align*}
such that $\langle Y\rangle_{\rm ws}$ is equal to the image of $X_1\times\{x_2\}$ in $S$, for some $x_2\in X_2$. 

Let $\Gamma_1:={\rm p}_1(\Gamma)$ and $\Gamma_2:={\rm p}_2(\Gamma)$, where ${\rm p}_1$ and ${\rm p}_2$ are the projections from $G$ to $G_1$ and $G_2$, respectively. Then $\Gamma':=\Gamma_1\times \Gamma_2$ is a congruence subgroup of $G(\QQ)_+$ containing $\Gamma$ as a finite index subgroup. Let $\phi:\Gamma\backslash X\rightarrow\Gamma'\backslash X$ denote the natural (finite) morphism. Then $\phi(W)\subseteq \phi(Y)\subseteq S':=\Gamma'\backslash X$, and we have
\begin{align*}
\dim \langle Y\rangle &=\dim \langle \phi(Y)\rangle,\\
\dim \langle W\rangle &=\dim \langle \phi(W)\rangle,\\
\dim \langle Y\rangle_{\rm ws}&=\dim \langle \phi(Y)\rangle_{\rm ws},\\
\dim \langle W\rangle_{\rm ws}&=\dim \langle \phi(W)\rangle_{\rm ws}.
\end{align*}
Therefore, after replacing $Y$, $W$, and $S$ by $\phi(Y)$, $\phi(W)$, and $S'$, respectively, we may assume that $\Gamma$ is of the form $\Gamma_1\times \Gamma_2$, and $S=\Gamma_1\backslash X_1\times \Gamma_2\backslash X_2=S_1\times S_2$. 

Thus, $\langle Y\rangle_{\rm ws}=S_1\times\{s_2\}$, where $s_2$ is the image of $x_2$ in $S_2$, $Y=Y_1\times \{s_2\}$, where $Y_1$ is the projection of $Y$ to $S_1$, and $W=W_1\times \{s_2\}$, where $W_1$ is the projection of $W$ to $S_1$. In particular, we can take 
\begin{align*}
x:=(x_1,x_2)\in X_1\times X_2
\end{align*}
such that $\langle W\rangle$ is equal to the image in $S$ of the $M(\RR)^+$ conjugacy class $X_M$ of $x$, where $M:=\MT(x)$. 

Again, there exists a decomposition
\begin{align*}
M^\ad=M_1\times M_2,
\end{align*}
which induces a splitting
\begin{align*}
X_M=X_{M_1}\times X_{M_2}\subseteq X=X_1\times X_2
\end{align*}
such that $\langle W\rangle_{\rm ws}$ is equal to the image in $S$ of $X_{M_1}\times\{y_2\}$, for some $y_2\in X_{M_2}$. 

Since $\MT(x_2)$ is equal to $G_2$, it follows that $M$ is a subgroup of $G_1\times G_2$ that surjects on to the second factor. In particular, 
\begin{align*}
X_M=M^{\rm der}(\RR)^+x
\end{align*}
surjects on to $X_2$. Therefore, let $M'_1$ and $M'_2$ be two normal semisimple subgroups of $M^{\rm der}$ corresponding to $M_1$ and $M_2$, respectively, so that
\begin{align*}
M^{\rm der}(\RR)^+x=M'_1(\RR)^+M'_2(\RR)^+x.
\end{align*}
Since $W$ is contained in $S_1\times\{s_2\}$, the projection of $M'_1$ to $G_2$ must be trivial. Hence, $M'_1(\RR)^+x$ is contained in $X_1\times\{x_2\}$ and we conclude that $M'_2(\RR)^+x$ surjects on to $X_2$. Since \begin{align*}
M'_2(\RR)^+x=\{y_1\}\times X_{M_2},
\end{align*}
for some $y_1\in X_{M_1}$, we have
\begin{align*}
\dim\langle W\rangle-\dim\langle W\rangle_{\rm ws}=\dim X_{M_2}\geq\dim X_2=\dim\langle Y\rangle-\dim\langle Y\rangle_{\rm ws},
\end{align*}
as required.
\end{proof}

\begin{corollary}[cf. \cite{hp:o-min}, Proposition 4.5]\label{owo}
Let $V$ be a subvariety of $S$. A subvariety of $V$ that is optimal in $V$ is weakly optimal in $V$.
\end{corollary}

\section{The hyperbolic Ax-Schanuel conjecture}\label{AS}
In this section, we formulate various conjectures about Shimura varieties that are analogous to the original Ax-Schanuel theorem from functional transcendence theory.
\begin{theorem}[cf. \cite{Ax71}, Theorem 1] \label{asorigin}
Let $f_1,...,f_n\in \CC[[t_1,...,t_m]]$ be power series that are $\QQ$-linearly independent modulo $\CC$. Then we have the following inequality
\[
{\trdeg}_{\CC} \CC (f_1,...,f_n,e(f_1),...,e(f_n))\geq n+{\rm rank} (\frac{\partial f_i}{\partial t_j})_{\begin{subarray}{l}i=1,\ldots,n\\j=1,\ldots, m\end{subarray}
}
\]
where $e(f)=e^{2\pi i f}\in\CC[[t_1,...,t_m]]$.
\end{theorem}

The following theorem is then an immediate corollary.

\begin{theorem}\label{wasorigin}
Let $f_1,...,f_n\in \CC[[t_1,...,t_m]]$ as above. Then
\begin{align*}
{\trdeg}_{\CC} \CC (f_1,...,f_n)+{\trdeg}_{\CC} \CC(e(f_1),...,e(f_n))\geq n+{\rm rank} (\frac{\partial f_i}{\partial t_j})_{\begin{subarray}{l}i=1,\ldots,n\\j=1,\ldots, m\end{subarray}}.
\end{align*}
\end{theorem}

Let $\pi$ denote the uniformization map
\begin{align*}
\CC^n\rightarrow(\CC^{\times})^n:(x_1,...,x_n)\mapsto (e(x_1),...,e(x_n))
\end{align*}
and let $D_n$ denote its graph in $\CC^n\times(\CC^{\times})^n$. We can rephrase Theorem \ref{asorigin} as follows. 

\begin{theorem}[cf. \cite{Tsi2015minimality}, Theorem 1.2]\label{asorigin2}
Let $V$ be a subvariety of $\CC^n\times (\CC^{\times})^n$ and let $U$ be an irreducible analytic component of $V\cap D_n$. Assume that the projection of $U$ to $(\CC^{\times})^n$ is not contained in a coset of a proper subtorus of $(\CC^{\times})^n$. Then
\begin{align*}
\dim V\geq \dim U+n.
\end{align*}
\end{theorem}

Similarly, we can rephrase Theorem \ref{wasorigin} as follows.

\begin{theorem} \label{wasorigin2}
Let $W$ be a subvariety of $\CC^n$ and $V$ a subvariety of $(\CC^{\times})^n$. Let $A$ be an irreducible analytic component of $W\cap \pi^{-1}(V)$. If $A$ is not contained in $b+L$, for any proper $\QQ$-linear subspace $L$ of $\CC^n$ and any $b\in\CC^n$, then
\begin{align*}
\dim V+\dim W\geq \dim A+n.
\end{align*}
\end{theorem}

Recall that $X$ is naturally endowed with the structure of a hermitian symmetric domain. In particular, it is a complex manifold. We define an (irreducible algebraic) {\bf subvariety} of $X$ as in Appendix B of \cite{kuy:ax-lindemann}. In particular, we consider the Harish-Chandra realization of $X$, which is a bounded domain in $\CC^N$, for some $N\in\NN$, and we define an (irreducible algebraic) subvariety of $X$ to be an irreducible analytic component of the intersection of $X$ with an algebraic subvariety of $\CC^N$. We define an (irreducible algebraic) subvariety of $X\times S$ to be an irreducible analytic component of the intersection of $X\times S$ with an algebraic subvariety of $\CC^N\times S$. We note, however, that, by \cite{kuy:ax-lindemann}, Corollary B.2, the algebraic structure that we are putting on $X$ and $X\times S$ does not depend on our particular choice of the Harish-Chandra realization of $X$; any realization of $X$ would yield the same algebraic structures. 

We are, therefore, able to formulate conjectures for Shimura varieties that are analogous to those above. Let $\pi$ henceforth denote the uniformization map
\begin{align*}
X\rightarrow S
\end{align*}
and let $D_S$ denote the graph of $\pi$ in $X\times S$. The following conjecture generalizes Conjecture 1.1 of \cite{pila2014ax}.

\begin{conjecture}[hyperbolic Ax-Schanuel]\label{hasc}
Let $V$ be a subvariety of $X\times S$ and let $U$ be an irreducible analytic component of $V\cap D_S$. Assume that the projection of $U$ to $S$ is not contained in a weakly special subvariety of $\Sh_K(G,\mathfrak{X})$ strictly contained in $S$. Then
\begin{align*}
\dim V\geq\dim U+\dim S.
\end{align*}
\end{conjecture}

For $S=\CC^n$, Conjecture \ref{hasc} and its generalization involving derivatives were obtained in \cite{pila2014ax}. Mok, Pila, and Tsimerman have very recently announced a proof of Conjecture \ref{hasc} in full \cite{MPT:AS}. 

For applications to the Zilber-Pink conjecture, only the following weaker version will be needed.

\begin{conjecture}[cf. \cite{hp:o-min}, Conjecture 5.10]\label{whasc}
Let $W$ be a subvariety of $X$ and let $V$ be a subvariety of $S$. Let $A$ be an irreducible analytic component of $W\cap \pi^{-1}(V)$ and assume that $\pi(A)$ is not contained in a weakly special subvariety of $\Sh_K(G,\mathfrak{X})$ strictly contained in $S$. Then
\begin{align*}
\dim V+\dim W\geq \dim A+\dim S.
\end{align*}
\end{conjecture}

\begin{proof}[Proof that Conjecture \ref{hasc} implies Conjecture \ref{whasc}]
Consider the situation described in the statement of Conjecture \ref{whasc}. Then $Y:=W\times V$ is an algebraic subvariety of $X\times S$ and
\begin{align*}
U:=\{(a,\pi(a));a\in A\}
\end{align*}
is an irreducible analytic component of $Y\cap D_S$. Clearly, the projection of $U$ to $S$ is not contained in a weakly special subvariety of $\Sh_K(G,\mathfrak{X})$ strictly contained in $S$. Therefore, by Conjecture \ref{hasc}, \begin{align*}
\dim Y\geq \dim U+\dim S
\end{align*}
and the result follows since $\dim U=\dim A$ and $\dim Y=\dim W+\dim V$.
\end{proof}

In our applications, we will use a reformulation of Conjecture \ref{whasc}. For this reformulation, we will need the following definitions. 

Fix a subvariety $V$ of $S$.

\begin{definition}
An {\bf intersection component} of $\pi^{-1}(V)$ is an irreducible analytic component of the intersection of $\pi^{-1}(V)$ with a subvariety of $X$.
\end{definition}

For any intersection component $A$ of $\pi^{-1}(V)$, there exists a smallest subvariety of $X$ containing $A$; we denote it $\langle A\rangle_{\rm Zar}$. It follows that $A$ is an irreducible analytic component of
\begin{align*}
\langle A\rangle_{\rm Zar}\cap\pi^{-1}(V).
\end{align*}

\begin{definition}
Let $A$ be an intersection component of $\pi^{-1}(V)$. We define the {\bf Zariski defect} of $A$ to be
\begin{align*}
\delta_{\rm Zar}(A):=\dim\langle A\rangle_{\rm Zar}-\dim A.
\end{align*}
\end{definition}

\begin{definition}
We say that an intersection component $A$ of $\pi^{-1}(V)$ is {\bf Zariski optimal} in $\pi^{-1}(V)$ if, for any intersection component $B$ of $\pi^{-1}(V)$, 
\begin{align*}
A\subsetneq B\subseteq\pi^{-1}(V)\implies\delta_{\rm Zar}(B)>\delta_{\rm Zar}(A).
\end{align*}
\end{definition}

\begin{definition}
Let $x\in X$ and let $X_M$ denote the $M(\RR)^+$ conjugacy class of $x$ in $X$, where $M:=\MT(x)$. Write $M^{\rm ad}$ as a product
\begin{align*}
M^{\rm ad}=M_1\times M_2
\end{align*}
of two normal $\QQ$-subgroups, either of which may be trivial, thus inducing a splitting
\begin{align*}
X_M=X_1\times X_2.
\end{align*}
For any $x_1\in X_1$ or $x_2\in X_2$, we obtain a subvariety $\{x_1\}\times X_2$ or $X_1\times\{x_2\}$ of $X$. We refer to any subvariety of $X$ taking this form as a {\bf pre-weakly special subvariety} of $X$. That is, a weakly special subvariety of $\Sh_K(G,\mathfrak{X})$ contained in $S$ is, by definition, the image in $S$ of a pre-weakly special subvariety of $X$.
\end{definition}

\begin{remark}
Note that pre-weakly special subvarieties of $X$ are indeed subvarieties of $X$ (see \cite{gao:AO}, Lemma 6.2, for example). In particular, they are irreducible analytic subsets of $X$. As explained in \cite{Moonen:linear1}, pre-weakly special subvarieties of $X$ are totally geodesic subvarieties of $X$.
\end{remark}

\begin{definition}
An intersection component $A$ of $\pi^{-1}(V)$ is called {\bf pre-weakly special} if $\langle A\rangle_{\rm Zar}$ is a pre-weakly special subvariety of $X$.
\end{definition}

\begin{conjecture}[weak hyperbolic Ax-Schanuel]\label{whasc2}
Let $A$ be an intersection component of $\pi^{-1}(V)$ that is Zariski optimal in $\pi^{-1}(V)$. Then $A$ is pre-weakly special.
\end{conjecture}

Note that Conjecture \ref{whasc2} is a direct generalization of the hyperbolic Ax-Lindemann theorem.

\begin{theorem}[hyperbolic Ax-Lindemann]\label{halw}
The maximal subvarieties contained in $\pi^{-1}(V)$ are pre-weakly special.
\end{theorem}

\begin{proof}[Proof that Conjecture \ref{whasc2} implies Theorem \ref{halw}]
The maximal subvarieties contained in $\pi^{-1}(V)$ are precisely the intersection components of $\pi^{-1}(V)$ that are Zariski optimal in $\pi^{-1}(V)$ and whose Zariski defect is zero.
\end{proof}

Although \cite{hp:o-min}, Section 5.2 is dedicated to products of modular curves, the proof that Formulations A and B of {\it Weak Complex Ax} are equivalent is completely general and, when translated into our terminology, yields the following.

\begin{lemma}
Conjecture \ref{whasc} and Conjecture \ref{whasc2} are equivalent.
\end{lemma}

We conclude this section with the following consequence of the weak hyperbolic Ax-Schanuel conjecture. Here and elsewhere, we will tacitly make use of the following remark.

\begin{remark}
Let $W$ be a subvariety of $S$ and let $A$ denote an irreducible analytic component of $\pi^{-1}(W)$ in $X$. Then, since $W$ is analytically irreducible, every irreducible analytic component of $\pi^{-1}(W)$ is equal to a $\Gamma$-translate of $A$ (as mentioned in \cite{ey:subvarieties}, Section 4, for example). In particular, $\pi(A)$ is equal to $W$.
\end{remark}

\begin{lemma}\label{ZOimpliesclosed}
Assume that the weak hyperbolic Ax-Schanuel conjecture is true for $V$. Let $A$ be a Zariski optimal intersection component of $\pi^{-1}(V)$. Then $\pi(A)$ is a closed irreducible subvariety of $V$ and, as such, is weakly optimal in $V$. 
\end{lemma}

\begin{proof}
Clearly, the Zariski closure $\overline{\pi(A)}$ of $\pi(A)$ is irreducible. Therefore, let $W$ be a subvariety of $V$ containing $\overline{\pi(A)}$ such that $\delta_{\rm ws}(W)\leq\delta_{\rm ws}(\overline{\pi(A)})$. We can and do assume that $W$ is weakly optimal in $V$. Let $B$ be an irreducible analytic component of $\pi^{-1}(W)$ containing $A$. We have
\begin{align*}
\delta_{\rm Zar}(B)&=\dim\langle B\rangle_{\rm Zar}-\dim B=\dim\langle B\rangle_{\rm Zar}-\dim W\\
&\leq\dim\langle W\rangle_{\rm ws}-\dim W=\delta_{\rm ws}(W)\leq\delta_{\rm ws}(\overline{\pi(A)})\\
&\leq\dim\langle A\rangle_{\rm Zar}-\dim A=\delta_{\rm Zar}(A),
\end{align*}
where we use the fact that, by the weak hyperbolic Ax-Schanuel conjecture, $\langle A\rangle_{\rm Zar}$ is pre-weakly special. Therefore, we conclude that $B=A$. Hence, $\pi(A)=\pi(B)=W$.

\end{proof}

\section{A finiteness result for weakly optimal subvarieties}\label{finiteweak}

In this section, we deduce from the weak hyperbolic Ax-Schanuel conjecture a finiteness statement for the weakly optimal subvarieties of a given subvariety $V$.

\begin{definition}
Let $x\in X$ and let $X_M$ denote the $M(\RR)^+$ conjugacy class of $x$ in $X$, where $M:=\MT(x)$. Then $X_M$ is a subvariety of $X$ and we refer to any subvariety of $X$ taking this form as a {\bf pre-special subvariety} of $X$. In particular, a pre-special subvariety of $X$ is a pre-weakly special subvariety of $X$. If $X_M$ is a point, that is, if $M$ is a torus, we refer to $X_M$ as a {\bf pre-special point} of $X$. A special subvariety of $\Sh_K(G,\mathfrak{X})$ contained in $S$ is, by definition, the image in $S$ of a pre-special subvariety of $X$.  
\end{definition}

\begin{definition}
Let $x\in X$ and let $X_M$ denote the $M(\RR)^+$ conjugacy class of $x$ in $X$, where $M:=\MT(x)$. Decomposing $M^{\rm ad}$ as a product
\begin{align*}
M^{\rm ad}=M_1\times M_2
\end{align*}
of two normal $\QQ$-subgroups, either of which may be trivial, induces a splitting
\begin{align*}
X_M=X_1\times X_2.
\end{align*}
For any such splitting, and any $x_1\in X_1$ or $x_2\in X_2$, we refer to the pre-weakly special subvariety $\{x_1\}\times X_2$ or $X_1\times\{x_2\}$ as a {\bf fiber of (the pre-special subvariety) $X_M$}. In particular, the points of $X_M$ are all fibers of $X_M$, and so too is $X_M$ itself.
\end{definition}

The main result of this section is the following.

\begin{proposition}[cf. \cite{hp:o-min}, Proposition 6.6]\label{fwo} Let $V$ be a subvariety of $S$ and assume that the weak hyperbolic Ax-Schanuel conjecture is true for $V$. Then there exists a finite set $\Sigma$ of pre-special subvarieties of $X$ such that the following holds.

Let $W$ be a subvariety of $V$ that is weakly optimal in $V$. Then there exists $Y\in\Sigma$ such that $\langle W\rangle_{\rm ws}$ is equal to the image in $S$ of a fiber of $Y$.
\end{proposition}

Note that similar theorems also hold for abelian varieties (see \cite{hp:o-min}, Proposition 6.1 and \cite{remond2009intersection}, Proposition 3.2).

Now fix a subvariety $V$ of $S$. Given an intersection component $A$ of $\pi^{-1}(V)$, there is a smallest totally geodesic subvariety $\langle A\rangle_{\rm geo}$ of $X$ that contains $A$. In particular, we may make the following definition.

\begin{definition}
Let $A$ be an intersection component of $\pi^{-1}(V)$. We define the {\bf geodesic defect} of $A$ to be
\begin{align*}
\delta_{\rm geo}(A):=\dim\langle A\rangle_{\rm geo}-\dim A.
\end{align*}
We note that, in \cite{hp:o-min}, this notion was referred to as the M\"obius defect of $A$.
\end{definition}

\begin{definition}
We say that an intersection component $A$ of $\pi^{-1}(V)$ is {\bf geodesically optimal} in $\pi^{-1}(V)$ if, for any intersection component $B$ of $\pi^{-1}(V)$, 
\begin{align*}
A\subsetneq B\subseteq \pi^{-1}(V)\implies\delta_{\rm geo}(B)>\delta_{\rm geo}(A).
\end{align*}
We note that the terminology geodesically optimal has a different meaning in \cite{hp:o-min}.
\end{definition}

\begin{remark} 
Let $A$ be an intersection component of $\pi^{-1}(V)$. If $A$ is geodesically optimal in $\pi^{-1}(V)$, then $A$ is an irreducible analytic component of
\begin{align*}
\langle A\rangle_{\rm geo}\cap\pi^{-1}(V).
\end{align*}
\end{remark}

\begin{lemma}\label{geo}
Assume that the weak hyperbolic Ax-Schanuel conjecture is true for $V$ and let $A$ be an intersection component of $\pi^{-1}(V)$. If $A$ is geodesically optimal in $\pi^{-1}(V)$, then $A$ is Zariski optimal in $\pi^{-1}(V)$.
\end{lemma}

\begin{proof}
Suppose that $B$ is an intersection component of $\pi^{-1}(V)$ containing $A$ such that
\begin{align*}
\delta_{\rm Zar}(B)\leq\delta_{\rm Zar}(A).
\end{align*}
We can and do assume that $B$ is Zariski optimal and so, by the weak hyperbolic Ax-Schanuel conjecture, it is pre-weakly special. In particular, $\langle B\rangle_{\rm Zar}$ is a pre-weakly special subvariety of $X$ and, therefore, equal to $\langle B\rangle_{\rm geo}$. Then
\begin{align*}
\delta_{\rm geo}(B)=\delta_{\rm Zar}(B)\leq\delta_{\rm Zar}(A)\leq\delta_{\rm geo}(A),
\end{align*}
and, since $A$ is geodesically optimal in $\pi^{-1}(V)$, we conclude that $B=A$.
\end{proof}

\begin{lemma}\label{upstairs}
Assume that the weak hyperbolic Ax-Schanuel conjecture is true for $V$ and let $W$ be a subvariety of $V$ that is weakly optimal in $V$. Let $A$ be an irreducible analytic component of $\pi^{-1}(W)$. Then $A$ is an intersection component of $\pi^{-1}(V)$ and is geodesically optimal in $\pi^{-1}(V)$.
\end{lemma}

\begin{proof}
Clearly, $A$ is an intersection component of $\pi^{-1}(V)$ since $W$ is an irreducible component of $\langle W\rangle_{\rm ws}\cap V$ and $\pi^{-1}\langle W\rangle_{\rm ws}$ is equal to the $\Gamma$-orbit of a pre-weakly special subvariety of $X$.

Therefore, let $B$ be an intersection component of $\pi^{-1}(V)$ containing $A$ such that
\begin{align*}
\delta_{\rm geo}(B)\leq\delta_{\rm geo}(A).
\end{align*}
We can and do assume that $B$ is geodesically optimal in $\pi^{-1}(V)$ and so, by Lemma \ref{geo}, $B$ is Zariski optimal in $\pi^{-1}(V)$. Therefore, by the weak hyperbolic Ax-Schanuel conjecture, $B$ is pre-weakly special i.e.  $\langle B\rangle_{\rm Zar}$ is a pre-weakly special subvariety of $X$. 

Let $Z:=\pi(B)$ (which is a closed irreducible subvariety of $V$ by Lemma \ref{ZOimpliesclosed}). We claim that $\langle Z\rangle_{\rm ws}=\pi(\langle B\rangle_{\rm Zar})$. To see this, note that $Z$ is contained in $\pi(\langle B\rangle_{\rm Zar})$ and so $\langle Z\rangle_{\rm ws}$ is contained in $\pi(\langle B\rangle_{\rm Zar})$. On the other hand, $\langle B\rangle_{\rm Zar}$ is contained in $\pi^{-1}(\langle Z\rangle_{\rm ws})$ and so $\pi(\langle B\rangle_{\rm Zar})$ is contained in $\langle Z\rangle_{\rm ws}$, which proves the claim. Therefore,
\begin{align*}
\delta_{\rm ws}(Z)&=\dim\pi(\langle B\rangle_{\rm Zar})-\dim Z\\&=\dim \langle B\rangle_{\rm Zar}-\dim Z\\
&=\dim \langle B\rangle_{\rm Zar}-\dim B=\delta_{\rm geo}(B)\leq\delta_{\rm geo}(A)\\
&\leq\dim\langle W\rangle_{\rm ws}-\dim W=\delta_{\rm ws}(W).
\end{align*}
Since $W$ is weakly optimal in $V$ and contained in $Z$, we conclude that $Z=W$. In particular, $B$ is contained in $\pi^{-1}(W)$ and, therefore, $B=A$.
\end{proof}

Let us briefly summarize the relationship between Zariski optimal and weakly optimal.

\begin{proposition}\label{ZOsummary}
Assume that the weak hyperbolic Ax-Schanuel conjecture is true for $V$. 

If $A$ is a Zariski optimal intersection component of $\pi^{-1}(V)$, then $\pi(A)$ is a closed irreducible subvariety of $V$ that is weakly optimal in $V$. 

On the other hand, if $W$ is a subvariety of $V$ that is weakly optimal in $V$, and $A$ is an irreducible analytic component of $\pi^{-1}(W)$, then $A$ is a Zariski optimal intersection component of $\pi^{-1}(V)$.
\end{proposition}

\begin{proof}
The first claim is Lemma \ref{ZOimpliesclosed}, whereas the second claim is Lemma \ref{upstairs} and Lemma \ref{geo}.
\end{proof}

As explained in \cite{kuy:ax-lindemann}, there exists an open semialgebraic fundamental set $\mathcal{F}$ in $X$ for the action of $\Gamma$ such that the set $\mathcal{V}:=\pi^{-1}(V)\cap\mathcal{F}$ is definable. 

\begin{definition}\label{fs}
Once and for all, let $\mathcal{F}$ denote an open semialgebraic fundamental set $\mathcal{F}$ in $X$ for the action of $\Gamma$, as above, and let $\mathcal{V}$ denote the definable set $\pi^{-1}(V)\cap\mathcal{F}$.
\end{definition}
 
Recall from \cite{vdD:tame-topology}, 1.17 that the {\bf local dimension} $\dim_xA$ of a definable set $A$ at a point $x\in A$ is definable. By \cite{hp:o-min}, Lemma 6.2, if $A$ is also a (complex) analytic set, then this dimension is exactly twice the local analytic dimension at $x$. Furthermore, if $A$ is analytically irreducible, then its local dimension at the points of $A$ is constant. For the remainder of this section, dimensions will be taken in the sense of definable sets. The key step in the proof of Proposition \ref{fwo} is the following. 
 
\begin{proposition}\label{geoprop}
Assume that the weak hyperbolic Ax-Schanuel theorem is true for $V$. There exists a finite set $\Sigma$ of pre-special subvarieties of $X$ such that the following holds.

Let $A$ be an intersection component of $\pi^{-1}(V)$ that is pre-weakly special such that, for some $x\in\langle A\rangle_{\rm Zar}\cap\mathcal{V}$,
\begin{align*}
\dim A=\dim_x(\langle A\rangle_{\rm Zar}\cap\mathcal{V}).
\end{align*} 
Then there exists $Y\in\Sigma$ such that $\langle A\rangle_{\rm Zar}$ is equal to a fiber of $Y$.
\end{proposition}

In order to prove Proposition \ref{geoprop}, we require some further preparations.

\begin{definition}
We say that a real semisimple algebraic group $ F$ is {\bf without compact factors} if it is equal to an almost direct product of almost simple subgroups whose underlying real Lie groups are not compact. We allow the product to be trivial i.e. we consider the trivial group as a real semisimple algebraic group without compact factors.
\end{definition}

\begin{lemma}\label{characterisation}
A subvariety of $X$ that is totally geodesic in $X$ is of the form
\begin{align*}
{ F}(\RR)^+ x,
\end{align*}
where $ F$ is a semisimple algebraic subgroup of ${ G}_{\RR}$ without compact factors and $x\in X$ factors through
\begin{align*}
{ G}_{ F}:={ F}{ Z}_{{ G}_{\RR}}({ F})^{\circ}.
\end{align*}
Conversely, if ${ F}$ is a semisimple algebraic subgroup of ${ G}_{\RR}$ without compact factors and $x\in X$ factors through ${ G}_{ F}$, then ${ F}(\RR)^+x$ is a subvariety of $X$ that is totally geodesic in $X$.
\end{lemma}

\begin{proof}
See \cite{uy:algebraic-flows}, Proposition 2.3.
\end{proof}

We let $\Omega$ denote a set of representatives for the ${ G}(\RR)$-conjugacy classes of semisimple algebraic subgroups of ${G}_{\RR}$ that are without compact factors. Note that $\Omega$ is a finite set (see \cite{BDR}, Corollary 0.2, for example), and it is clear that the set
\begin{align*}
\Pi_0:=\{(x,g,F)\in \mathcal{V}\times{ G}(\RR)\times\Omega:x(\mathbb{S})\subseteq g{ G}_{{ F}}g^{-1}\},
\end{align*}
parametrising (albeit in a many-to-one fashion) the totally geodesic subvarieties of $X$ passing through $\mathcal{V}$, is definable. Consider the two functions
\begin{align*}
d(x,g,F)&:=\dim_x(g{ F}(\RR)^+g^{-1}x)=\dim(g{ F}(\RR)^+g^{-1}x)\\
d_{\mathcal{V}}(x,g,F)&:=\dim_x(\mathcal{V}\cap g{ F}(\RR)^+g^{-1}x),
\end{align*}
and let $\Pi_1$ denote the definable set
\begin{align*}
\{(x,g,F)\in\Pi_0:&\ (x,g_1,F_1)\in\Pi_0,\ gF(\RR)^+g^{-1} x\subsetneq g_1F_1(\RR)^+g_1^{-1} x\\
& \implies d(x,g,F)-d_{\mathcal{V}}(x,g,F)<d(x,g_1,F_1)-d_{\mathcal{V}}(x,g_1,F_1)\}.
\end{align*}
Finally, let $\Pi_2$ denote the definable set
\begin{align*}
\{(x,g,F)\in\Pi_1:&\ (x,g_1,F_1)\in\Pi_0,\ g_1F_1(\RR)^+g_1^{-1} x\subsetneq gF(\RR)^+g^{-1} x\\
& \implies d_{\mathcal{V}}(x,g_1,F_1)<d_{\mathcal{V}}(x,g,F)\}.
\end{align*}

The proof of Proposition \ref{geoprop} will require the following three lemmas.

\begin{lemma}\label{lem3}
Let $A$ be an intersection component of $\pi^{-1}(V)$ that is pre-weakly special such that, for some $x\in\langle A\rangle_{\rm Zar}\cap\mathcal{V}$,
\begin{align*}
\dim A=\dim_x(\langle A\rangle_{\rm Zar}\cap\mathcal{V}).
\end{align*}  
Then we can write
\begin{align*}
\langle A\rangle_{\rm Zar}=g{ F}(\RR)^+g^{-1} x,
\end{align*}
where $(x,g,F)\in\Pi_2$.
\end{lemma}
\begin{proof}
By Lemma \ref{characterisation}, we can write
\begin{align*}
\langle A\rangle_{\rm Zar}=g{ F}(\RR)^+g^{-1} x
\end{align*}
for some $F\in\Omega$ and some $x\in\mathcal{V}$ that factors through $gG_Fg^{-1}$. In particular, $(x,g,F)\in\Pi_0$. By assumption, we can and do choose $x\in\langle A\rangle_{\rm Zar}\cap\mathcal{V}$ such that
\begin{align*}
\dim A=\dim_x(\langle A\rangle_{\rm Zar}\cap\mathcal{V})=d_{\mathcal{V}}(x,g,F).
\end{align*}

Suppose that $(x,g,F)$ does not belong to $\Pi_1$ i.e. that there exists $(x,g_1,F_1)\in\Pi_0$ such that 
\begin{align*}
gF(\RR)^+g^{-1} x\subsetneq g_1F_1(\RR)^+g_1^{-1} x,
\end{align*}
and
\begin{align}\label{ineq}
d(x, g,F)-d_{\mathcal{V}}(x,g,F)\geq d(x,g_1,F_1)-d_{\mathcal{V}}(x,g_1,F_1).
\end{align}
Let $B$ be an irreducible analytic component of 
\begin{align*}
g_1F_1(\RR)^+g_1^{-1} x\cap\pi^{-1}(V)
\end{align*}	
passing through $x$ such that
\begin{align*}
\dim B=d_{\mathcal{V}}( x,g_1,F_1).
\end{align*}
From (\ref{ineq}), we obtain $\delta_{\rm Zar}(B)\leq\delta_{\rm Zar}( A)$.

On the other hand, the Intersection Inequality (see \cite{gr:sheaves}, Chapter 5, \S3) yields
\begin{align*}
\dim_{x} (B\cap \langle A\rangle_{\rm Zar})\geq\dim B+\dim \langle A\rangle_{\rm Zar}-d( x,g_1,F_1)
\end{align*}
and, from (\ref{ineq}), we obtain
\begin{align*}
\dim_{ x}( B\cap  \langle A\rangle_{\rm Zar})\geq\dim A.
\end{align*}
It follows that $B\cap\langle A\rangle_{\rm Zar}$, and hence $B$ itself, contains a complex neighbourhood of $ x$ in $A$, which implies that $A$ is contained in $B$.

Since $A$ is Zariski optimal, we conclude that $ A=B$. However, this implies that
\begin{align*}
d( x,g_1,F_1)-d_{\mathcal{V}}( x,g_1,F_1)>2\delta_{\rm Zar}(B)=2\delta_{\rm Zar}( A)=d( x, g, F)-d_{\mathcal{V}}( x, g,F),
\end{align*}
which contradicts (\ref{ineq}). Therefore, $(x,g,F)\in\Pi_1$.

Now suppose that $( x, g,F)$ does not belong to $\Pi_2$ i.e. that there exists $( x,g_1,F_1)\in\Pi_0$ such that 
\begin{align*}
g_1F_1(\RR)^+g_1^{-1} x\subsetneq gF(\RR)^+g^{-1} x,
\end{align*}
and
\begin{align*}
d_{\mathcal{V}}(x,g_1,F_1)=d_{\mathcal{V}}( x, g,F)=\dim  A.
\end{align*}
But then $ A$ is contained in
\begin{align*}
g_1F_1(\RR)^+g_1^{-1} x\subsetneq\langle A\rangle_{\rm Zar},
\end{align*}
which contradicts the definition of $\langle A\rangle_{\rm Zar}$.
\end{proof}

\begin{lemma}\label{countable}
Assume that the weak hyperbolic Ax-Schanuel conjecture is true for $V$. Then, if $(x,g,F)\in\Pi_2$, there exists a semisimple subgroup $F'$ of $G$ defined over $\QQ$ such that $g{ F}g^{-1}$ is equal to the almost direct product of the almost simple factors of $F'_{\RR}$ whose underlying real Lie groups are non-compact. 
\end{lemma}

\begin{proof}
By \cite{ullmo:applications}, Proposition 3.1, it suffices to show that $g{ F}(\RR)^+g^{-1} x$ is a pre-weakly special subvariety of $X$. Therefore, let $A$ be an irreducible analytic component of
\begin{align*}
g{ F}(\RR)^+g^{-1} x\cap\pi^{-1}(V)
\end{align*}
passing through $x$ such that
\begin{align*}
\dim A=d_{\mathcal{V}}(x,g,F).
\end{align*}	

Let $B$ be an intersection component of $\pi^{-1}(V)$ containing $A$ such that $\delta_{\rm Zar}(B)\leq\delta_{\rm Zar}(A)$. We can and do assume that $B$ is Zariski optimal and, therefore, by the weak hyperbolic Ax-Schanuel conjecture, pre-weakly special i.e. $B$ is an irreducible component of
\begin{align*}
\langle B\rangle_{\rm Zar}\cap\pi^{-1}(V)
\end{align*}
and $\langle B\rangle_{\rm Zar}$ is a pre-weakly special subvariety of $X$. 

Therefore, $A$ is contained in
\begin{align*}
g{ F}(\RR)^+g^{-1} x\cap\langle B\rangle_{\rm Zar}
\end{align*}
and we let $Y$ be an irreducible analytic component of this intersection containing $A$. Then $Y$ is a subvariety of $X$ that is totally geodesic is $X$ and, hence, equal to $g_1F_1(\RR)^+g_1^{-1}x$ for some $(x,g_1,F_1)\in\Pi_0$. Furthermore,
\begin{align*}
d_{\mathcal{V}}(x,g_1,F_1)=\dim A=d_{\mathcal{V}}(x,g,F)
\end{align*}
and, since $(x,g,F)\in\Pi_2$, we conclude that 
\begin{align*}
g_1F_1(\RR)^+g_1^{-1}x=gF(\RR)^+g^{-1}x\subseteq\langle B\rangle_{\rm Zar}.
\end{align*}
We also have
\begin{align*}
\dim \langle B\rangle_{\rm Zar}-\dim_x(\langle B\rangle_{\rm Zar}\cap\mathcal{V})&\leq\delta_{\rm Zar}(B)\leq\delta_{\rm Zar}(A)\\
&\leq d(x,g,F)-\dim A\\
&=d(x,g,F)-d_{\mathcal{V}}(x,g,F),
\end{align*}
and so, since $(x,g,F)\in\Pi_1$, we conclude that
\begin{align*}
g{ F}(\RR)^+g^{-1} x=\langle B\rangle_{\rm Zar}.
\end{align*}
\end{proof}

\begin{lemma}\label{finiteset}
Assume that the weak hyperbolic Ax-Schanuel conjecture is true for $V$. Then, the set
\begin{align*}
\{g{ F}g^{-1}:(x,g,F)\in\Pi_2\}
\end{align*}
is finite.
\end{lemma}
\begin{proof}
Decompose $\Pi_2$ as the finite union of the $\Pi_F$, varying over the members $F$ of $\Omega$, where $\Pi_F$ denotes the set of tuples $(x,g,F)\in\Pi_2$. For each $F\in\Omega$, consider the map
\begin{align*}
\Pi_F\rightarrow{ G}(\RR)/{ N}_{{ G}(\RR)}({ F}),
\end{align*}		
defined by
\begin{align*}
(x,g,F)\mapsto g{ N}_{{ G}(\RR)}({ F}),
\end{align*}
whose image, therefore, is in bijection with $\{g{ F}g^{-1}:(x,g,F)\in\Pi_2\}$. It is also definable and, by Lemma \ref{countable}, it is countable. Hence, it is finite.
\end{proof}

\begin{proof}[Proof of Proposition \ref{geoprop}]
Let $A$ be an intersection component of $\pi^{-1}(V)$ that is pre-weakly special such that, for some $x\in\langle A\rangle_{\rm Zar}\cap\mathcal{V}$,
\begin{align*}
\dim A=\dim_x(\langle A\rangle_{\rm Zar}\cap\mathcal{V}).
\end{align*}
Then, by Lemma \ref{lem3}, we can write
\begin{align*}
\langle A\rangle_{\rm Zar}=gF(\RR)^+g^{-1} x
\end{align*}
where $(x,g,F)\in\Pi_2$. By Lemma \ref{countable}, there exists a semisimple subgroup $F'$ of $G$ defined over $\QQ$ such that $g{ F}g^{-1}$ is equal to the almost direct product of the almost simple factors of $F'_{\RR}$ whose underlying real Lie groups are non-compact. In fact, by \cite{ullmo:applications}, Proposition 3.1, $F'$ is the smallest subgroup of $G$ defined over $\QQ$ containing $gFg^{-1}$. Since, by Lemma \ref{finiteset}, $gFg^{-1}$ comes from a finite set, so too does $F'$. Therefore, the reductive algebraic group
\begin{align*}
M:=F' Z_G(F')^{\circ}
\end{align*}
is defined over $\QQ$ and belongs to a finite set. 

If we write $M^{\rm nc}$ for the almost direct product of the almost $\QQ$-simple factors of $M$ whose underlying real Lie groups are not compact, then $x$ factors through $M'_{\RR}:=Z(M)^{\circ}_{\RR} M^{\rm nc}_{\RR}$ and, if we write $\mathfrak{X}_M$ for the $M'(\RR)$ conjugacy class of $x$ in $\mathfrak{X}$, then, by \cite{ullmo:equidistribution}, Lemme 3.3, $(M',\mathfrak{X}_M)$ is a Shimura subdatum of $(G,\mathfrak{X})$. Furthermore, by \cite{uy:andre-oort}, Lemma 3.7, the number of Shimura subdatum $(M',\mathfrak{Y})$ is finite. Therefore, since the $M'(\RR)^+$ conjugacy class $X_M$ of $x$ in $X$ is a pre-special subvariety of $X$ and $\langle A\rangle_{\rm Zar}$ is a fiber of $X_M$, the proof is complete.
\end{proof}

\begin{proof}[Proof of Proposition \ref{fwo}]
Let $A$ be an irreducible analytic component of $\pi^{-1}(W)$. By Proposition \ref{ZOsummary}, $A$ is an intersection component of $\pi^{-1}(V)$ and is Zariski optimal in $\pi^{-1}(V)$. Therefore, by the weak hyperbolic Ax-Schanuel conjecture, $A$ is pre-weakly special. It follows that the image of $\langle A\rangle_{\rm Zar}$ in $S$ is equal to $\langle W\rangle_{\rm ws}$.

After possibly replacing $A$ by a $\gamma A$, for some $\gamma\in\Gamma$, we can and do assume that there exists $x\in\langle A\rangle_{\rm Zar}\cap\mathcal{V}$ such that
\begin{align*}
\dim(A)=\dim_x(\langle A\rangle_{\rm Zar}\cap\mathcal{V}).
\end{align*}
By Proposition \ref{geoprop}, $\langle A\rangle_{\rm Zar}$ is a fiber of $Y\in\Sigma$, where $\Sigma$ is a finite set of pre-special subvarieties of $X$ depending only on $V$.
\end{proof}

\section{Anomalous subvarieties}

In this section, we recall the notion of an anomalous subvariety, which is defined by Bombieri, Masser, and Zannier in \cite{bombieri2007anomalous} for subvarieties of algebraic tori. In fact, we give the more general notion of an $r$-anomalous subvariety, as introduced by R\'emond \cite{remond2009intersection}. 

Let $V$ be a subvariety of $S$. We will use Proposition \ref{fwo} to show that, under the weak hyperbolic Ax-Schanuel conjecture, the union of the subvarieties of $V$ that are $r$-anomalous in $V$ constitutes a Zariski closed subset of $V$. We will then give a criterion for when it is a proper subset.

\begin{definition}
Let $r\in \ZZ$. A subvariety $W$ of $V$ is called {\bf $r$-anomalous} in $V$ if
\begin{align*}
\dim W\geq \max\{1,r+\dim \langle W\rangle_{\rm ws}-\dim S\}.
\end{align*}
 A subvariety of $V$ is {\bf maximal $r$-anomalous} in $V$ if it is $r$-anomalous in $V$ and not strictly contained in another subvariety of $V$ that is also $r$-anomalous in $V$. 
\end{definition}

We denote by ${\rm an}(V,r)$ the set of subvarieties of $V$ that are maximal $r$-anomalous in $V$ and by $V^{{\rm an},r}$ the union of the elements of ${\rm an}(V,r)$, which is then the union of all the subvarieties of $V$ that are $r$-anomalous in $V$.

We say that a subvariety of $V$ is {\bf anomalous} if it is $(1+\dim V)$-anomalous. We write ${\rm an}(V)$ for ${\rm an}(V,1+\dim V)$ and $V^{{\rm an}}$ for $V^{{\rm an},1+\dim V}$.

\begin{theorem} \label{openan}
Assume that the weak hyperbolic Ax-Schanuel conjecture is true for $V$ and let $r\in\ZZ$. Then $V^{{\rm an},r}$ is a Zariski closed subset of $V$. 
\end{theorem}

We refer the reader to \cite{bombieri2007anomalous}, \cite{remond2009intersection}, and \cite{hp:o-min} for similar results on algebraic tori and abelian varieties. We will require the following facts.

\begin{proposition}[cf. \cite{hartshorne1977algebraic}, Chapter 2, Exercise 3.22 (d)]\label{dimension}
Let $f:W\rightarrow Y$ be a dominant morphism between two integral schemes of finite type over a field and let 
\begin{align*}
e:=\dim W-\dim Y
\end{align*}
denote the {\bf relative dimension}. For $h\in\ZZ$, let $E_h$ denote the set of points $x\in W$ such that the fibre $f^{-1}(f(x))$ possesses an irreducible component of dimension at least $h$ that contains $x$. Then 
\begin{enumerate}
\item[(1)] $E_h$ is a Zariski closed subset of $W$,
\item[(2)] $E_e=W$, and
\item[(3)] if $h>e$, $E_h$ is not Zariski dense in $W$.

\end{enumerate}
\end{proposition}

\begin{lemma}\label{anomisopt}
Let $W\in{\rm an}(V,r)$. Then $W$ is weakly optimal in $V$.
\end{lemma}

\begin{proof}
Let $Y$ be a subvariety of $V$ containing $W$ such that $\delta_{\rm ws}(Y)\leq \delta_{\rm ws}(W)$. We can and do assume that $Y$ is weakly optimal. Then
\begin{align*}
\dim Y &=\dim \langle Y\rangle_{\rm ws}-\delta_{\rm ws}(Y)\\&\geq \dim \langle Y\rangle_{\rm ws}-\delta_{\rm ws}(W)\\
&=\dim \langle Y\rangle_{\rm ws}-(\dim \langle W\rangle_{\rm ws}-\dim W)\\
 &\geq \dim \langle Y\rangle_{\rm ws}+r-\dim S.
\end{align*}
Since $Y$ contains $W$, we know that $\dim Y\geq 1$, and so $Y$ is $r$-anomalous in $V$. Since $W$ is maximal $r$-anomalous in $V$, we conclude that $Y$ must be equal to $W$. Therefore, $W$ is weakly optimal.
\end{proof}

\begin{proof}[Proof of Theorem \ref{openan}]
Let $\Sigma$ be a finite set of pre-special subvarieties of $X$ (whose existence is ensured by Proposition \ref{fwo}) such that, if $W$ is a subvariety of $V$ that is weakly optimal in $V$, then there exists $x\in X$ such that, if $M:=\MT(x)$, the $M(\RR)^+$ conjugacy class $X_M$ of $x$ in $X$ belongs to $\Sigma$ and $\langle W \rangle_{\rm ws}$ is equal to the image in $S$ of a fiber of $X_M$. That is, we may write $M^{\ad}$ as a product $M_1\times M_2$ of two normal $\QQ$-subgroups, which induces a splitting $X=X_1\times X_2$, such that $\langle W \rangle_{\rm ws}$ is equal to the image in $S$ of $\{x_1\}\times X_2$, for some $x_1\in X_1$. 

Let $W\in{\rm an}(V,r)$. By Lemma \ref{anomisopt}, there exists $X_M\in\Sigma$ such that $\langle W \rangle_{\rm ws}$ is equal to the image in $S$ of $\{x_1\}\times X_2$, for some $x_1\in X_1$, where $X_M=X_1\times X_2$, as above. 

Let $\Gamma_M$ be a congruence subgroup of $M(\QQ)_+$ contained in $\Gamma$, where $M(\QQ)_+$ denotes the subgroup of $M(\QQ)$ acting on $X_M$, and let $\Gamma_1$ denote the image of $\Gamma$ under the natural maps
\begin{align*}
M(\QQ)\rightarrow M^{\ad}(\QQ)\rightarrow M_1(\QQ).
\end{align*}
We denote by $f$ the restriction of
\begin{align*}
\Gamma_M\backslash X_M\rightarrow\Gamma_1\backslash X_1
\end{align*}
to an irreducible component $\widetilde{V}$ of $\phi^{-1}(V)$, such that $\dim\widetilde{V}=\dim V$, where $\phi$ denotes the natural map
\begin{align*}
\Gamma_M\backslash X_M\rightarrow\Gamma\backslash X=S.
\end{align*}
In particular, $\phi(\widetilde{V})=V$.
Therefore, by Proposition \ref{dimension} (1), the set $E_h$ of points $z$ in $\widetilde{V}$ such that the fibre $f^{-1}(f(z))$ possesses an irreducible component of dimension at least $h\in\ZZ$ that contains $z$ is a Zariski closed subset of $\widetilde{V}$. Since $\phi$ is a closed morphism, $\phi(E_h)$ is Zariski closed in $V$. 

We claim that $W$ is contained in $\phi(E_h)$, where
\begin{align*}
h:=\max\{1,r+ \dim X_2-\dim S\}.
\end{align*}
To see this, fix an irreducible component $\widetilde{W}$ of $\phi^{-1}(W)$ contained in $\widetilde{V}$ such that $\dim\widetilde{W}=\dim W$. Then $\langle \widetilde{W} \rangle_{\rm ws}$ is equal to the image of $\{x_1\}\times X_2$ in $\Gamma_M\backslash X_M$ and so $\widetilde{W}$ lies in a fiber of $f$. Since
\begin{align*}
\dim \widetilde{W}=\dim W\geq\max\{1,r+\dim \langle W \rangle_{\rm ws}-\dim S\}=\max\{1,r+\dim X_2-\dim S\},
\end{align*}
$\widetilde{W}$ is contained in $E_h$, which implies that $W$ is contained in $\phi(E_h)$.

On the other hand, we claim that $\phi(E_h)$ is contained in $V^{{\rm an},r}$. To see this, let $z\in E_h$ and let $Y$ be an irreducible component of the fibre $f^{-1}(f(z))$ of dimension at least $h$ containing $z$. Then $Y$ is contained in the image of $\{x_1\}\times X_2$ in $\Gamma_M\backslash X_M$, where $x_1\in X_1$ lies above $f(z)\in\Gamma_1\backslash X_1$, and so
\begin{align*}
\dim \langle Y \rangle_{\rm ws}\leq\dim X_2.
\end{align*}
Therefore,
\begin{align*}
\dim\phi(Y)=\dim Y\geq h=\max\{1,r+ \dim X_2-\dim S\}\geq\max\{1,r+\dim\langle \phi(Y) \rangle_{\rm ws}-\dim S\}
\end{align*}
and so $\phi(Y)$ is $r$-anomalous in $V$.

Hence, if we let $E$ denote the union of the $\phi(E_h)$ as we vary over the finitely many maps $f$ obtained from the $X_M\in\Sigma$ and their possible splittings, we conclude that $E=V^{{\rm an},r}$, which finishes the proof.
\end{proof}

We denote by $V^{\rm oa}$ the complement in $V$ of $V^{\rm an}$. By Theorem \ref{openan}, this is a (possibly empty) open subset of $V$. In the literature, it is sometimes referred to as the open-anomalous locus, hence the subscript.
We conclude this section by showing that, when $V$ is sufficiently generic, $V^{\rm oa}$ is not empty.

\begin{proposition}\label{equal}
Suppose that $V$ is Hodge generic in $S$. Then $V^{\rm an}=V$ if and only if we can write $G^\ad=G_1\times G_2$, and thus $X=X_1\times X_2$, such that
\begin{align*}
\dim f(V)<{\rm min}\{\dim V,\dim X_1\},
\end{align*}
where $f$ denotes the projection map
\begin{align*}
\Gamma\backslash X\rightarrow\Gamma_1\backslash X_1,
\end{align*}
and $\Gamma_1$ denotes the image of $\Gamma$ under the natural maps
\begin{align*}
G(\QQ)\rightarrow G^\ad(\QQ)\rightarrow G_1(\QQ).
\end{align*}
\end{proposition}

\begin{proof}
First suppose that $V^{\rm an}=V$. Then, for any set $\Sigma$ as in the proof of Theorem \ref{openan}, $V$ is contained in the (finite) union of the images in $S$ of the $X_M\in\Sigma$. Therefore, since $V$ is assumed to be Hodge generic in $S$, it must be that $X\in\Sigma$ and, furthermore, that there exists $W\in{\rm an}(V)$ such that $G^\ad=G_1\times G_2$, and thus $X=X_1\times X_2$, such that $\langle W\rangle_{\rm ws}$ is equal to the image in $S$ of $\{x_1\}\times X_2$, for some $x_1\in X_1$. 

Let $f$ denote the projection map
\begin{align*}
\Gamma\backslash X\rightarrow\Gamma_1\backslash X_1
\end{align*}
and consider its restriction
\begin{align*}
V\rightarrow\overline{f(V)},
\end{align*}
where $\overline{f(V)}$ denotes the Zariski closure of $f(V)$ in $\Gamma_1\backslash X_1$. Since $V^{\rm an}=V$, it follows from Proposition \ref{dimension} (3), that
\begin{align*}
h:=\max\{1,1+\dim V+ \dim X_2-\dim X\}\leq\dim V-\dim f(V).
\end{align*}
Hence,
\begin{align*}
\dim f(V)<\dim X-\dim X_2=\dim X_1
\end{align*}
and
\begin{align*}
\dim f(V)\leq\dim V-h\leq\dim V-1<\dim V.
\end{align*}

Conversely, suppose that $G^\ad=G_1\times G_2$, and thus $X=X_1\times X_2$, such that
\begin{align*}
\dim f(V)<{\rm min}\{\dim V,\dim X_1\},
\end{align*}
where $f$ again denotes the projection map
\begin{align*}
\Gamma\backslash X\rightarrow\Gamma_1\backslash X_1.
\end{align*}
Restricting $f$ to
\begin{align*}
V\rightarrow\overline{f(V)},
\end{align*}
as before, we see from Proposition \ref{dimension} (2) that the set $E_h$ of points $z$ in $V$ such that the fibre $f^{-1}(f(z))$ possesses an irreducible component of dimension at least 
\begin{align*}
h:=\max\{1,1+\dim V-\dim X_1\}\leq \dim V-\dim f(V)=\dim V-\dim\overline{f(V)}
\end{align*}
that contains $z$ is equal to $V$. However, from the proof of Theorem \ref{openan}, we have seen that $E_h$ is contained in $V^{\rm an}$, so the claim follows.
\end{proof}

\begin{corollary}\label{coran}
If $G^\ad$ is $\QQ$-simple and $V$ is a Hodge generic subvariety in $S$, then $V^{\rm an}$ is strictly contained in $V$. In particular, $V^{\rm an}$ is strictly contained in $V$ whenever $V$ is a Hodge generic subvariety of $\mathcal{A}_g$.
\end{corollary}

\section{Main results (part 1): Reductions to point counting}\label{reductions}

In this section, we prove our main theorem: under the weak hyperbolic Ax-Schanuel conjecture, the Zilber-Pink conjecture can be reduced to a problem of point counting. We also give a reduction of Pink's conjecture in the case when the open-anomalous locus is non-empty. 

\begin{definition}
Let $V$ be a subvariety of $S$. We denote by $\Opt_0(V)$ the set of all points in $V$ that are optimal in $V$.
\end{definition}

Consider the following corollary of the Zilber-Pink conjecture.

\begin{conjecture}\label{zerodim}
Let $V$ be a subvariety of $S$. Then $\Opt_0(V)$ is finite.
\end{conjecture}

We will later show that, under certain arithmetic hypotheses, one can prove Conjecture \ref{zerodim} when $V$ is a curve. Our main result in this section is that (under the weak hyperbolic Ax-Schanuel conjecture), Conjecture \ref{zerodim} implies the Zilber-Pink conjecture.

\begin{theorem}\label{main theorem}
Assume that the weak hyperbolic Ax-Schanuel conjecture is true and assume that Conjecture \ref{zerodim} holds. 

Let $V$ be a subvariety of $S$. Then $\Opt(V)$ is finite.
\end{theorem}

\begin{proof}
We prove Theorem \ref{main theorem} by induction on $\dim V$. Of course, Theorem \ref{main theorem} is trivial when $\dim V=0$ or $\dim V=1$. Therefore, we assume that $\dim V\geq 2$ and that Theorem \ref{main theorem} holds whenever the subvariety in question is of lower dimension.

We need to show that the induction hypothesis implies that there are only finitely many subvarieties of positive dimension belonging to $\Opt(V)$.

Let $\Sigma$ be a finite set of pre-special subvarieties of $X$, as in the proof of Theorem \ref{openan}, and let $W\in\Opt (V)$ be of positive dimension. 

By Corollary \ref{owo}, $W$ is weakly optimal and, therefore, there exists $x\in X$ such that, if $M:=\MT(x)$, the $M(\RR)^+$ conjugacy class $X_M$ of $x$ in $X$ belongs to $\Sigma$ and $\langle W \rangle_{\rm ws}$ is equal to the image in $S$ of a fiber of $X_M$. That is, we may write $M^{\ad}$ as a product
\begin{align*}
M^{\ad}=M_1\times M_2
\end{align*} 
of two normal $\QQ$-subgroups, thus inducing a splitting
\begin{align*}
X_M=X_1\times X_2,
\end{align*} 
such that $\langle W \rangle_{\rm ws}$ is equal to the image in $S$ of $\{x_1\}\times X_2$, for some $x_1\in X_1$.

Let $\Gamma_M$ be a congruence subgroup of $M(\QQ)_+$ contained in $\Gamma$, where $M(\QQ)_+$ denotes the subgroup of $M(\QQ)$ acting on $X_M$, such that the image of $\Gamma_M$ under the natural map
\begin{align*}
M(\QQ)\rightarrow M^{\ad}(\QQ)=M_1(\QQ)\times M_2(\QQ)
\end{align*}
is equal to a product $\Gamma_1\times\Gamma_2$. We denote by $f$ the natural morphism
\begin{align*}
\Gamma_M\backslash X_M\rightarrow\Gamma_1\backslash X_1,
\end{align*}
and by $\phi$ the finite morphism
\begin{align*}
\Gamma_M\backslash X_M\rightarrow\Gamma\backslash X=S.
\end{align*}
Let $\widetilde{V}$ be an irreducible component of $\phi^{-1}(V)$ such that $\dim\widetilde{V}=\dim V$, and let $\widetilde{W}$ denote an irreducible component of $\phi^{-1}(W)$ contained in $\widetilde{V}$ such that $\dim\widetilde{W}=\dim W$. Then $\widetilde{W}$ is optimal in $\widetilde{V}$. On the other hand, by the generic smoothness property, there exists a dense open subset $V_0$ of $\widetilde{V}$ such that the restriction $f_0$ of $f$ to $V_0$ is a smooth morphism of relative dimension $\nu$. We denote by $V_1$ the Zariski closure of $f(V_0)$ in $\Gamma_1\backslash X_1$. 

Now suppose that
\begin{align}\label{openset}
\widetilde{W}\cap V_0=\emptyset.
\end{align}
Then $\widetilde{W}$ is a subvariety of some irreducible component $V^0$ of $\widetilde{V}\setminus V_0$. Furthermore, $\widetilde{W}$ is optimal in $V^0$. However, since $\dim V^0$ is strictly less than $\dim V$, our induction hypothesis implies that $\Opt(V^0)$ is finite. 

Therefore, we assume that (\ref{openset}) does not hold. As an irreducible component of the fibre $f_0^{-1}(z)$, where $z$ denotes the image of $x_1$ in $V_1$, its dimension is equal to $\nu$. In particular,
\begin{align*}
\dim\widetilde{W}=\nu.
\end{align*}
We claim that $z$ is optimal in $V_1$. To see this, note that $f(\langle \widetilde{W}\rangle)$ contains $z$ and is a special subvariety of dimension
\begin{align*}
\dim\langle \widetilde{W}\rangle-\dim X_2=\dim\widetilde{W}+\delta(\widetilde{W})-\dim X_2=\nu+\delta(\widetilde{W})-\dim X_2.
\end{align*}
Therefore, let $A$ be a subvariety of $V_1$ containing $z$ such that 
\begin{align*}
\delta(A)\leq \delta(z)\leq\nu+\delta(\widetilde{W})-\dim X_2,
\end{align*}
and let $B$ be an irreducible component of $f^{-1}(A)$ containing $\widetilde{W}$. Since $V_0$ is open in $\widetilde{V}$ and $A$ is contained in $V_1$,
\begin{align*}
\dim B=\dim A+\nu.
\end{align*}
Therefore,
\begin{align*}
\delta(B)\leq\dim\langle A\rangle+\dim X_2-\dim B=\delta(A)+\dim A+\dim X_2-\dim B\leq\delta(\widetilde{W})
\end{align*}
and, since $\widetilde{W}$ is optimal in $\widetilde{V}$, we conclude that $B$ is equal to $\widetilde{W}$. In particular, $\widetilde{W}$ is an irreducible component of $f^{-1}(A)$ but, since it is also contained in $f^{-1}(z)$, it must be that $A$ is equal to $z$, proving the claim.

Since $W$ was assumed to be of positive dimension, so too must be $X_2$. It follows that $\dim V_1$ is strictly less than $\dim V$ and so, by the induction hypothesis, $\Opt(V_1)$ is finite. Since $z\in\Opt(V_1)$ and since $\Sigma$ and the number of splittings are finite, we are done.

\end{proof}

We will later prove that the following conjecture is a consequence of the weak hyperbolic Ax-Schanuel conjecture and our arithmetic conjectures. It is inspired by the cited theorem of Habegger and Pila.

\begin{conjecture}[cf. \cite{hp:o-min}, Theorem 9.15 (iii)]\label{oafinite}
Let $V$ be a subvariety of $S$. Then the set $V^{\rm oa}\cap S^{[1+\dim V]}$ is finite.
\end{conjecture}

The importance of Conjecture \ref{oafinite} for us is that, when $V$ is suitably generic, Conjecture \ref{oafinite} implies Pink's conjecture (assuming the weak hyperbolic Ax-Schanuel conjecture).

\begin{theorem}
Assume that the weak hyperbolic Ax-Schanuel conjecture is true and that Conjecture \ref{oafinite} holds.

Let $V$ be a Hodge generic subvariety of $S$ such that (even after replacing $\Gamma$) $S$ cannot be decomposed as a product $S_1\times S_2$ such that $V$ is contained in $V'\times S_2$, where $V'$ is a proper subvariety of $S_1$ of dimension strictly less than the dimension of $V$. Then 
\begin{align*}
V\cap S^{[1+\dim V]}
\end{align*}
is not Zariski dense in $V$.
\end{theorem}

\begin{proof}
We claim that the assumptions guarantee that $V^{\rm an}$ is strictly contained in $V$. Otherwise, by Proposition \ref{equal}, we can write $G^\ad=G_1\times G_2$, and thus $X=X_1\times X_2$, such that
\begin{align*}
\dim f(V)<{\rm min}\{\dim V,\dim X_1\},
\end{align*}
where $f$ denotes the projection map
\begin{align*}
\Gamma\backslash X\rightarrow\Gamma_1\backslash X_1,
\end{align*}
and $\Gamma_1$ denotes the image of $\Gamma$ under the natural maps
\begin{align*}
G(\QQ)\rightarrow G^\ad(\QQ)\rightarrow G_1(\QQ).
\end{align*}
Therefore, after replacing $\Gamma$, we can write $S$ as a product $S_1\times S_2$ so that $f$ is simply the projection on to the first factor and $V$ is contained in $V'\times S_2$, where $V'$ is Zariski closure of $f(V)$ in $S_1$. However, since
\begin{align*}
\dim V'=\dim f(V),
\end{align*}
this is a contradiction.

Therefore, by Theorem \ref{openan}, $V^{\rm an}$ is a proper Zariski closed subset of $V$. On the other hand, $V\cap S^{[1+\dim V]}$ is contained in
\begin{align*}
V^{\rm an}\cup \left[V^{\rm oa}\cap S^{[1+\dim V]}\right]
\end{align*}
and so the theorem follows from Conjecture \ref{oafinite}.
\end{proof}

\section{The counting theorem}\label{count}

Henceforth, we turn our attention to the counting problems themselves. We will approach these problems using a theorem of Pila and Wilkie concerned with counting points in definable sets. We first recall the notations. 

Let $k\geq 1$ be an integer. For any real number $y$, we define its {\bf $k$-height} as
\begin{align*}
{\rm H}_k(y):=\min\{\max\{|a_0|,...,|a_k|\}:a_i \in \ZZ,\ {\rm gcd}\{a_0,...,a_k\}=1,\ a_0y^k+...+a_k=0\},
\end{align*}
where we use the convention that, if the set is empty i.e. $y$ is not algebraic of degree at most $k$, then ${\rm H}_k(y)$ is $+\infty$. For $y=(y_1,...,y_m)\in \RR^m$, we set
\begin{align*}
{\rm H}_k(y):=\max\{{\rm H}_k(y_1),...,{\rm H}_k(y_m)\}.
\end{align*}

For any set $A\subseteq\RR^m\times \RR^n$, and for any real number $T\geq 1$, we define
\begin{align*}
A(k,T):=\{(y,z)\in Z:{\rm H}_k(y)\leq T\}.
\end{align*}

The counting theorem of Pila and Wilkie is stated as follows. 

\begin{theorem}[cf. the proof of \cite{hp:o-min}, Corollary 7.2]\label{pilawilkie}
Let $D\subseteq \RR^l\times\RR^m\times \RR^n$ be a definable family parametrised by $\RR^l$, let $k$ be a positive integer, and let $\epsilon>0$. There exists a constant $c:=c(D,k,\epsilon)>0$ with the following properties. 

Let $x\in\RR^l$ and let 
\begin{align*}
D_x:=\{(y,z)\in\RR^m\times\RR^n:(x,y,z)\in D\}.
\end{align*}
Let $\pi_1$ and $\pi_2$ denote the projections $\RR^m\times \RR^n\rightarrow\RR^m$ and $\RR^m\times \RR^n\rightarrow\RR^n$, respectively. If $T\geq 1$ and $\Sigma\subseteq D_x(k,T)$ satisfies
\begin{align*}
\# \pi_2(\Sigma)>cT^{\epsilon},
\end{align*}
there exists a continuous and definable function $\beta: [0,1]\rightarrow D_x$ such that the following properties hold.
\begin{enumerate}
  \item  The composition $\pi_1\circ \beta: [0,1]\rightarrow \RR^m$ is semi-algebraic.
  \item  The composition $\pi_2\circ \beta:[0,1]\rightarrow \RR^n$ is non-constant.
  \item  We have $\beta(0)\in\Sigma$.
  \item  The restriction $\beta_{|(0,1)}$ is real analytic.
\end{enumerate}
\end{theorem}
Note that, although the conclusion $\beta(0)\in\Sigma$ does not appear in the statement of \cite{hp:o-min}, Corollary 7.2, it is, indeed, established in its proof. The final property holds because $\RR_{\rm an,exp}$ admits analytic cell decomposition (see \cite{vdDM:Ranexp}).

\section{Complexity}

In order to apply the counting theorem, we will need a way of counting special points and, more generally, special subvarieties. Recall that $S$ is a connected component of the Shimura variety $\Sh_K(G,\mathfrak{X})$ defined by the Shimura datum $(G,X)$ and the compact open subgroup $K$ of $G(\AAA_f)$.

Let $P$ be a special point in $S$ and let $x\in X$ be a pre-special point lying above $P$. In particular, $T:=\MT(x)$ is a torus and we denote by $D_T$ the absolute value of the discriminant of its splitting field. We let $K^m_T$ denote the maximal compact open subgroup of $T(\AAA_f)$ and we let $K_T\subseteq K^m_T$ denote $K\cap T(\AAA_f)$.

\begin{definition}
The {\bf complexity} of $P$ is the natural number
\begin{align*}
\Delta(P):=\max\{D_T,[K^m_T:K_T]\}.
\end{align*}
Note that this does not depend on the choice of $x$.
\end{definition}

Now let $Z$ be a special subvariety of $S$. There exists a Shimura subdatum $({ H},{\mathfrak{X}_H})$ of $({ G},{\mathfrak{X}})$, such that $H$ is the generic Mumford-Tate group on $\mathfrak{X}_H$, and a connected component $X_{ H}$ of ${\mathfrak{X}_H}$ contained in $X$ such that $Z$ is the image of $X_{ H}$ in $\Gamma\backslash X$. In fact, these choices are well-defined up to conjugation by $\Gamma$. 

By the {\bf degree} $\deg(Z)$ of $Z$, we refer to the degree (in the sense of \cite{KY:AO}, Section 5.1) of the Zariski closure of $Z$ in the Baily-Borel compactification of $S$ with respect to the line bundle $L_K$ defined in \cite{KY:AO}, Proposition 5.3.2 (1).

\begin{definition}
The {\bf complexity} of $Z$ is the natural number
\begin{align*}
\Delta(Z):=\max\{\deg(Z),\text{min}\{\Delta(P):P\in Z\text{ is a special point}\}\}.
\end{align*}
Note that when $Z$ is a special point, this complexity coincides with the former.
\end{definition}

This is a natural generalization of the complexities given in \cite{hp:o-min}, Definition 3.4 and Definition 3.8. In order to count special subvarieties, however, it is crucial that the complexity of $Z$ satisfies the following property.

\begin{conjecture} \label{count complexity}
For any $b\geq 1$, we have
\begin{align*}
\#\{Z\subseteq S:Z\text{ is special and }\Delta(Z)\leq b\}<\infty.
\end{align*}
\end{conjecture}

The obstruction to proving that this property holds for a general Shimura variety can be expressed as follows.

\begin{conjecture}\label{degcom}
For any $b\geq 1$, there exists a finite set $\Omega$ of semisimple subgroups of $G$ defined over $\QQ$ such that, if $Z$ is a special subvariety of $S$, and $\deg(Z)\leq b$, then
\begin{align*}
H^{\rm der}=\gamma F\gamma^{-1},
\end{align*}
for some $\gamma\in\Gamma$ and some $F\in\Omega$.
\end{conjecture}

We will later verify Conjecture \ref{degcom} for a product of modular curves.

\begin{proof}[Proof that Conjecture \ref{degcom} implies Conjecture \ref{count complexity}]
Let $Z$ be a special subvariety of $S$ such that $\Delta(Z)\leq b$. By Conjecture \ref{degcom}, there exists a finite set $\Omega$ of semisimple subgroups of $G$ defined over $\QQ$, independent of $Z$, such that
\begin{align*}
H^{\rm der}=\gamma F\gamma^{-1},
\end{align*}
for some $\gamma\in\Gamma$ and some $F\in\Omega$.

Let $P\in Z$ be a special point such that $\Delta(P)$ is minimal among all special points in $Z$ and let $x\in X$ be a point lying above $P$ such that $\MT(x)$ is contained in $H$. Therefore, $Z$ is equal to the image of $F(\RR)^+\gamma^{-1}x$ in $\Gamma\backslash X$. Furthermore, $\MT(\gamma^{-1}x)$ is contained in
\begin{align*}
G_F:=F Z_G(F)^{\circ}
\end{align*}
and, by \cite{ullmo:equidistribution}, Lemme 3.3, if we denote by $\mathfrak{X}'$ the $G_F(\RR)$ conjugacy class of $\gamma^{-1}x$, we obtain a Shimura subdatum $(G_F,\mathfrak{X}')$ of $(G,\mathfrak{X})$. 

Therefore, let $X'$ denote the connected component $G_F(\RR)^+\gamma^{-1}x$ of $\mathfrak{X}'$ and let $\Gamma'$ denote $\Gamma\cap G_F(\QQ)_+$, where $G_F(\QQ)_+$ denotes the subgroup of $G_F(\QQ)$ acting on $X'$. By \cite{uy:andre-oort}, Proposition 3.21 and its proof, there exist only finitely many $\Gamma'$ orbits of pre-special points in $X'$ whose image in $\Gamma'\backslash X'$ has complexity at most $b$. Therefore, there exists $\lambda\in\Gamma'$ such that $\gamma^{-1}x=\lambda y$, where $y\in X'$ belongs to a finite set. We conclude that $Z$ is equal to the image of
\begin{align*}
\Gamma F(\RR)^+\gamma^{-1}x=\Gamma F(\RR)^+\lambda y=\Gamma \lambda F(\RR)^+y=\Gamma F(\RR)^+y
\end{align*}
in $\Gamma\backslash X$, which concludes the proof.
\end{proof}

\section{Galois orbits}

In \cite{hp:o-min}, Habegger and Pila formulated a conjecture about Galois orbits of optimal points in $\CC^n$ that in \cite{hp:beyond} they had been able to prove for so-called asymmetric curves. In \cite{orr:unlikely}, Orr generalized the result to asymmetric curves in $\mathcal{A}^2_g$. 

Recall that $\Sh_K(G,\mathfrak{X})$ possesses a canonical model, defined over a number field $E$, which depends only on $(G,\mathfrak{X})$. Furthermore, $S$ is defined over a finite abelian extension $F$ of $E$. In particular, for any extension $L$ of $F$ contained in $\CC$, it makes sense to say that a subvariety $V$ of $S$ is defined over $L$. Moreover, if $V$ is such a subvariety, then ${\Aut}(\CC/L)$ acts on the points of $V$.

If $Z$ is a special subvariety of $S$ and $\sigma\in{\Aut}(\CC/F)$, then $\sigma(Z)$ is also a special subvariety of $S$ and its complexity is also $\Delta(Z)$. In particular, if $V$ is a subvariety of $S$, as above, then ${\Aut}(\CC/L)$ acts on $\Opt(V)$ and its orbits are finite.

\begin{conjecture}[large Galois orbits]\label{LGO}
Let $V$ be a subvariety of $S$, defined over a finitely generated extension $L$ of $F$ contained in $\CC$. There exist positive constants $c_G$ and $\delta_G$ such that the following holds. 

If $P\in\Opt_0(V)$, then
\begin{align*}
\#{\Aut}(\CC/L)\cdot P\geq c_G\Delta(\langle P\rangle)^{\delta_G}.
\end{align*}
\end{conjecture}

\begin{remark}
In the context of the Andr\'{e}-Oort conjecture, there is the pioneering hypothesis that Galois orbits of special points should be large. See \cite{basopen}, Problem 14 for the formulation for special points in $\mathcal{A}_g$ and see \cite{yafaevduke}, Theorem 2.1 for special points in a general Shimura variety. This hypothesis, which was verified by Tsimerman for special points of $\mathcal{A}_g$ \cite{tsimerman:AO} via progress on the Colmez conjecture due to Andreatta, Goren, Howard, Madapusi Pera \cite{AGHM:colmez} and Yuan and Zhang \cite{YZ:colmez}, is now the only obstacle in an otherwise unconditional proof of the Andr\'{e}-Oort conjecture. The conjecture is that there exist positive constants $c$ and $\delta$ such that, for any special point $P\in S$,
\begin{align*}
\#{\Gal}(\bar{\QQ}/\QQ)\cdot P\geq c\Delta(P)^\delta.
\end{align*}

Of course, this conjecture does not follow from Conjecture \ref{LGO} because special points lying in $V$ need not be optimal in $V$. However, the proof of the Andr\'{e}-Oort conjecture only requires the bound for special points that are not contained in the positive dimensional special subvarieties contained in $V$ i.e. special points contained in $\Opt_0(V)$ (see \cite{daw:book} for more details). Furthermore, since special points are defined over number fields, we may also assume in that case that $V$ is defined over a finite extension of $F$.  It follows that Conjecture \ref{LGO} is sufficient to prove the Andr\'e-Oort conjecture.
\end{remark}

To prove Conjecture \ref{oafinite}, however, one only requires the following hypothesis.

\begin{conjecture}\label{LGOoa}
Let $V$ be a subvariety of $S$, defined over a finitely generated extension $L$ of $F$ contained in $\CC$. There exist positive constants $c_G$ and $\delta_G$ such that the following holds. 

If $P\in V^{\rm oa}\cap S^{[1+\dim V]}$, then
\begin{align*}
\#{\Aut}(\CC/L)\cdot P\geq c_G\Delta(\langle P\rangle)^{\delta_G}.
\end{align*}
\end{conjecture}

\begin{remark}\label{oainopt}
Note that, if $P\in V^{\rm oa}\cap S^{[1+\dim V]}$, then $P\in\Opt_0(V)$. To see this, let $W$ be a subvariety of $V$ containing $P$ such that $\delta(W)\leq\delta(P)$ i.e.
\begin{align*}
\dim\langle W\rangle-\dim W\leq\dim\langle P\rangle\leq\dim S-1-\dim V.
\end{align*}
In fact, we can and do assume that $W$ is optimal. We have
\begin{align*}
\dim W\geq 1+\dim V+\dim\langle W\rangle-\dim S\geq 1+\dim V+\dim\langle W\rangle_{\rm ws}-\dim S,
\end{align*}
and so $\dim W=0$, as $P\notin V^{\rm an}$, which implies that $W=P$, proving the claim. Therefore, Conjecture \ref{LGOoa} follows from Conjecture \ref{LGO}, but the former may turn out to be more tractable. It is worth recalling that, when $S$ is an abelian variety and $V$ is a subvariety defined over $\bar{\QQ}$, Habegger \cite{habegger:abelian} famously showed that the N\'eron-Tate height is bounded on $\bar{\QQ}$-points of $V^{\rm oa}\cap S^{[\dim V]}$.
\end{remark}

\section{Further arithmetic hypotheses}\label{fah}

The principal obstruction to applying the Pila-Wilkie counting theorems to our point counting problems (except for the availability of lower bounds for Galois orbits) is the ability to parametrize pre-special subvarieties of $S$ using points of bounded height in a definable set.

\begin{definition}
We say that a semisimple algebraic group defined over $\QQ$ is of non-compact type if its almost-simple factors all have the property that their underlying real Lie group is not compact. 
\end{definition}

Let $\Omega$ be a set of representatives for the semisimple subgroups of $G$ defined over $\QQ$ of non-compact type modulo the equivalence relation 
\begin{align*}
H_2\sim H_1\iff H_{2,\RR}=gH_{1,\RR}g^{-1}\text{ for some }g\in G(\RR).
\end{align*}
Then $\Omega$ is a finite set (see \cite{BDR}, Corollary 0.2, for example). Add the trivial group to $\Omega$. Recall that we realise $X$ as a bounded symmetric domain in $\CC^N$ for some $N\in\NN$, which we identify with $\RR^{2N}$. We fix an embedding of $G$ into $\GL_n$ such that $\Gamma$ is contained in $\GL_n(\ZZ)$. We consider $\GL_n(\RR)$ as a subset of $\RR^{n^2}$ in the natural way. Recall the definition of $\mathcal{F}$ (Definition \ref{fs}).

\begin{conjecture}[cf. \cite{hp:o-min}, Proposition 6.7]\label{conj}
There exist positive constants $d$, $c_{\mathcal{F}}$, and $\delta_{\mathcal{F}}$ such that, if $z\in\mathcal{F}$, then the smallest pre-special subvariety of $X$ containing $z$ can be written $gF(\RR)^+g^{-1}x$, where $F\in\Omega$, and $g\in G(\RR)$ and $x\in X$ satisfy 
\begin{align*}
{\rm H}_d(g,x)\leq c_{\mathcal{F}}\Delta(\langle \pi(z)\rangle)^{\delta_{\mathcal{F}}}.
\end{align*}
\end{conjecture}

This is seemingly a natural generalization of the following theorem due to Orr and the first author on the heights of pre-special points, which plays a crucial role in the proof of the Andr\'e-Oort conjecture.

\begin{theorem}[cf. \cite{Daw2016}, Theorem 1.4]\label{daworr}  
There exist positive constants $d$, $c_{\mathcal{F}}$ and $\delta_{\mathcal{F}}$ such that, if $z\in\mathcal{F}$ is a pre-special, then
\begin{align*}
{\rm H}_d(z)\leq c_{\mathcal{F}}\Delta(\pi(z))^{\delta_{\mathcal{F}}}.
\end{align*}
\end{theorem}

We remark that the problem of finding $d$ as in Conjecture \ref{conj} poses no obstacle in itself. Indeed a proof of the following theorem will appear in a forthcoming article of Borovoi and the authors.

\begin{theorem}[cf. \cite{BDR}, Corollary 0.7]
There exists a positive constant $d$ such that, for any two semisimple subgroups $H_1$ and $H_2$ of $G$ defined over $\QQ$ that are conjugate by an element of $G(\RR)$, there exists a number field $K$ contained in $\RR$ of degree at most $d$, and an element $g\in G(K)$, such that 
\begin{align*}
H_{2,K}=gH_{1,K}g^{-1}.
\end{align*}
\end{theorem}

A nice feature of Conjecture \ref{conj} is that it implies Conjecture \ref{count complexity} that there are only finitely many special subvarieties of bounded complexity.

\begin{lemma}\label{conjimplic}
Conjecture \ref{conj} implies Conjecture \ref{count complexity}.
\end{lemma}

\begin{proof}
Let $Z$ be a special subvariety of $S$ such that $\Delta(Z)\leq b$ and let $P\in Z$ be such that $\langle P\rangle=Z$. Let $z\in\mathcal{F}$ be such that $\pi(z)=P$ and let $X_H$ be the smallest pre-special subvariety of $X$ containing $z$. Then $\pi(X_H)=Z$ and, by Conjecture \ref{conj}, $X_H=gFg^{-1}x$, where $F\in\Omega$, and $g\in G(\RR)$ and $x\in X$ satisfy 
\begin{align*}
{\rm H}_d(g,x)\leq c_{\mathcal{F}}\Delta(Z)^{\delta_{\mathcal{F}}}\leq c_{\mathcal{F}}b^{\delta_{\mathcal{F}}}.
\end{align*}
The claim follows, therefore, from the fact that there are only finitely many algebraic numbers of bounded degree and height i.e. Northcott's property.

\end{proof}

Another, albeit longer, approach to our point counting problems can be given by replacing Conjecture \ref{conj} with two related conjectures, although we will have to additionally assume Conjecture \ref{count complexity} in this case. We will also rely on the fact that Theorem \ref{pilawilkie} is uniform in families. The advantage is that the following two conjectures are seemingly more accessible.

\begin{conjecture}\label{fieldofdef}
For any $\kappa>0$, there exists a positive constant $c_\kappa$ such that, if $Z$ is a special subvariety of $S$, then there exists a semisimple subgroup $H$ of $G$ defined over $\QQ$ of non-compact type, and an extension $L$ of $F$ satisfying
\begin{align*}
[L:F]\leq c_\kappa\Delta(Z)^\kappa,
\end{align*}
such that, for any $\sigma\in{\Aut}(\CC/L)$,
\begin{align*}
\sigma(Z)=\pi(H(\RR)^+x_{\sigma}),
\end{align*}
where $H(\RR)^+x_\sigma$ is a pre-special subvariety of $X$ intersecting $\mathcal{F}$.
\end{conjecture}

Recall that, for an abelian variety $A$, defined over a field $K$, every abelian subvariety of $A$ can be defined over a fixed, finite extension of $K$. The analogue of Conjecture \ref{fieldofdef} is, therefore, trivial. In a Shimura variety, one hopes that the degrees of fields of definition of strongly special subvarieties grow as in Conjecture \ref{fieldofdef}. If this were true, Conjecture \ref{fieldofdef} for strongly special subvarieties would follow easily.

Our final conjecture is also inspired by the abelian setting.

\begin{conjecture}[cf. \cite{hp:o-min}, Lemma 3.2]\label{orr2}
There exist positive constants $c_{\Gamma}$ and $\delta_{\Gamma}$ such that, if $X_H$ is a pre-special subvariety of $X$ intersecting $\mathcal{F}$ and $z\in\mathcal{F}$ belongs to $\Gamma X_H$, then $z\in\gamma X_H$, where $\gamma\in\Gamma$ satisfies
\begin{align*}
{\rm H}_1(\gamma)\leq c_{\Gamma}\deg(\pi(X_H))^{\delta_{\Gamma}}.
\end{align*}
\end{conjecture}

Conjecture \ref{orr2} has the following useful consequence.

\begin{lemma}\label{heightx}
Assume that Conjecture \ref{orr2} holds.

There exist positive constants $d$, $c_{\rm H}$, and $\delta_{\rm H}$ such that, if $H(\RR)^+x$ is a pre-special subvariety of $X$ intersecting $\mathcal{F}$, then
\begin{align*}
H(\RR)^+x=H(\RR)^+y
\end{align*}
where ${\rm H}_d(y)\leq c_{\rm H}\Delta(\pi(H(\RR)^+x))^{\delta_{\rm H}}$.
\end{lemma}

\begin{proof}
Let $d$, $c_\mathcal{F}$, and $\delta_\mathcal{F}$ be the positive constants afforded to us by Theorem \ref{daworr}, and let $c_{\Gamma}$ and $\delta_{\Gamma}$ be the positive constants afforded to us by Conjecture \ref{orr2}. 

Let $x'\in\Gamma H(\RR)^+x\cap\mathcal{F}$ denote a pre-special point such that $\pi(x')$ is of minimal complexity among the special points of $\pi(H(\RR)^+x)$. By Theorem \ref{daworr}, we have
\begin{align*}
{\rm H}_d(x')\leq c_\mathcal{F}\Delta(\pi(H(\RR)^+x))^{\delta_\mathcal{F}}.
\end{align*}
On the other hand, by Conjecture \ref{orr2}, $x'\in\gamma H(\RR)^+x$, where
\begin{align*}
{\rm H}_1(\gamma)\leq c_\Gamma\Delta(\pi(H(\RR)^+x))^{\delta_\Gamma}.
\end{align*}
It follows easily from the properties of heights that there exist positive constants $c$ and $\delta$ depending only on the fixed data such that
\begin{align*}
{\rm H}_d(\gamma^{-1}x')\leq c{\rm H}_1(\gamma)^\delta{\rm H}_d(x')^\delta.
\end{align*}
Therefore, the previous remarks show that 
\begin{align*}
y:=\gamma^{-1}x'\in H(\RR)^+x
\end{align*}
satisfies the requirements of the lemma.
\end{proof}

We will now verify the arithmetic conjectures stated above in an arbitrary product of modular curves.

\section{Products of modular curves}

Our definition of a Shimura variety allows for the possibility that $S$ might be a product of modular curves. In that case $G=\GL^n_2$, where $n$ is the number of modular curves, and $\mathfrak{X}$ is the $G(\RR)$ conjugacy class of the morphism $\SSS\rightarrow G_{\RR}$ given by
\begin{align*}
a+ib\mapsto\left(\left( \begin{array}{cc}
a & b \\
-b & a \end{array} \right),\ldots,\left( \begin{array}{cc}
a & b \\
-b & a \end{array} \right)\right).
\end{align*}
We let $X$ denote the $G(\RR)^+$ conjugacy class of this morphism, which one identifies with the $n$-th cartesian power $\HH^n$ of the upper half-plane $\mathbb{H}$. 

For our purposes, we can and do suppose that $\Gamma$ is equal to $\SL_2(\ZZ)^n$ and we let $\mathcal{F}$ denote a fundamental set in $X$ for the action of $\Gamma$, equal to the $n$-th cartesian power of a fundamental set $\mathcal{F}_{\HH}$ in $\HH$ for the action of $\SL_2(\ZZ)$. Note that, as explained in \cite{orr:siegel}, Section 1.3, we can and do choose $\mathcal{F}_{\HH}$ in the image of a Siegel set. Via the $j$-function applied to each factor of $\HH^n$, the quotient $\Gamma\backslash X$ is isomorphic to the algebraic variety $\CC^n$. Special subvarieties have the following well-documented description.

\begin{proposition}[cf. \cite{edix:mod}, Proposition 2.1]
Let $I=\{1,\ldots, n\}$. A subvariety $Z$ of $\CC^n$ is a special subvariety if and only if there exists a partition $\Omega=(I_1,\ldots,I_t)$ of $I$, with $|I_i|=n_i$, such that $Z$ is equal to the product of subvarieties $Z_i$ of $\CC^{n_i}$, where, either
\begin{itemize}
\item $I_i $ is a one element set and $Z_i$ is a special point, or
\item $Z_i$ is the image of $\mathbb{H}$ in $\CC^{n_i}$ under the map sending $\tau\in\mathbb{H}$ to the image of $(g_j\tau)_{j\in I_i}$ in $\CC^{n_i}$ for elements $g_j\in\GL_2(\QQ)^+$.
\end{itemize}
\end{proposition}

First note that Conjecture \ref{conj} for $\CC^n$ follows from Proposition 6.7 of \cite{hp:o-min}. Hence, we will now verify Conjecture \ref{fieldofdef} and Conjecture \ref{orr2} in that setting.

\begin{proof}[Proof of Conjecture \ref{fieldofdef} for $\CC^n$]
Let $Z$ be a special subvariety of $\CC^n$, equal to a product of special subvarieties $Z_i$ of $\CC^{n_i}$, as above. Without loss of generality, we may assume that the product contains only one factor and, by Theorem \ref{daworr}, we may assume that it is not a special point. Therefore, $Z$ is equal to the image of $\mathbb{H}$ in $\CC^n$ under the map sending $\tau\in\mathbb{H}$ to the image of $(g_j\tau)_{j=1}^n$ in $\CC^{n}$ for elements $g_j\in\GL_2(\QQ)^+$.

In other words, we have a morphism of Shimura data from $(\GL_2,\HH^{\pm})$ to $(G,\mathfrak{X})$, where $\HH^{\pm}$ is the union of the upper and lower half-planes (or, rather, the conjugacy class we associate with it, as above), induced by the morphism
\begin{align*}
\GL_2\rightarrow\GL_2^n:g\mapsto(g_jgg^{-1}_j)_{j=1}^n,
\end{align*}
such that $Z$ is equal to the image of $\HH\times\{1\}$ under the corresponding morphism
\begin{align}\label{shimmap}
\Sh_{K}(\GL_2,\HH^{\pm})\rightarrow\Sh_{\GL_2(\hat{\ZZ})^n}(G,\mathfrak{X}),
\end{align}
where $K$ is the product of the groups
\begin{align*}
K_p:=g_1\GL_2(\ZZ_p)g^{-1}_1\cap\cdots\cap g_n\GL_2(\ZZ_p)g^{-1}_n
\end{align*}
over all primes $p$.

Since (\ref{shimmap}) is defined over $E(\GL_2,\HH^{\pm})=\QQ$, it suffices to bound the size of
\begin{align*}
\pi_0(\Sh_{K}(\GL_2,\HH^{\pm})),
\end{align*}
which, by \cite{milne:intro}, Theorem 5.17, is in bijection with
\begin{align*}
\QQ_{>0}\backslash \AAA^\times_f/\nu(K),
\end{align*}
where $\nu$ is the determinant map on $\GL_2$. However, since $\AAA^\times_f$ is equal to the direct product $\QQ_{>0}\hat{\ZZ}^\times$, it suffices to bound the size of $\hat{\ZZ}^\times/\nu(K)$.

To that end, let $\Sigma$ denote the (finite) set of primes $p$ such that $g_j\notin\GL_2(\ZZ_p)$, for some $j\in\{1,\dots,n\}$. In particular,
\begin{align*}
[\hat{\ZZ}^\times:\nu(K)]=\prod_{p\in\Sigma}[\ZZ^\times_p:\nu(K_p)]
\end{align*}
and, since $K_p$ contains the elements ${\rm diag}(a,a)$, where $a\in\ZZ^\times_p$, 
\begin{align*}
[\ZZ^\times_p:\nu(K_p)]\leq [\ZZ^\times_p:\ZZ^{\times 2}_p]\leq 4,
\end{align*}	
where $\ZZ^{\times 2}_p$ denotes the squares in $\ZZ^\times_p$.

On the other hand, by \cite{daw:aomod},
\begin{align*}
\deg(Z)\geq\prod_{p\in\Sigma}p,
\end{align*}
and the conjecture follows easily from the following classical fact regarding primorials.
\begin{lemma}
Let $n\in\NN$. The product of the first $n$ prime numbers is equal to
\begin{align*}
e^{(1+o(1))n\log n}.
\end{align*}

\end{lemma}

\end{proof}

\begin{proof}[Proof of Conjecture \ref{orr2} for $\CC^n$]
Let $X_H$ be a product of spaces $X_i\subseteq\HH^{n_i}$ each equal to either a pre-special point or to the image of $\HH$ given by the map sending $\tau$ to $(g_j\tau)_{j=1}^{n_i}$ for elements $g_j\in\GL_2(\QQ)^+$. Without loss of generality, we may assume that the product contains only one factor. If $X_H$ is a pre-special point contained in $\mathcal{F}$, then the claim follows from the fact that
\begin{align*}
\{\gamma\in\Gamma:\gamma\mathcal{F}\cap\mathcal{F}\neq\emptyset\}
\end{align*}
is finite. Therefore, assume that $X$ is equal to the image of $\HH$ in $\HH^{n}$ given by the map sending $\tau$ to $(g_j\tau)_{j=1}^{n}$ for elements $g_j\in\GL_2(\QQ)^+$. We can and do assume that $g_1$ is equal to the identity element and that all of the $g_j$ have coprime integer entries. 

As in the statement of Conjecture \ref{orr2}, we assume that $X$ intersects $\mathcal{F}$, and we let $x\in\mathcal{F}\cap X$. Therefore,
\begin{align*}
x=(g_j\tau_x)_{j=1}^n,
\end{align*}
where $\tau_x\in\mathcal{F}_{\HH}$. By \cite{orr:siegel}, Theorem 1.2 (cf. \cite{hp:beyond}, Lemma 5.2), ${\rm H}_1(g_j)\leq c_1\det(g_j)^2$, for all $j\in\{1,\ldots,n\}$, where $c_1$ is a positive constant not depending on $Z$. In particular,
\begin{align*}
{\rm H}_1(g^{-1}_j)\leq\det(g_j) {\rm H}_1(g_j)\leq c_1\det(g_j)^3,
\end{align*}	
for each $j\in\{1,\ldots,n\}$.

Now let $z:=(z_j)^n_{j=1}\in\mathcal{F}$ be a point belonging to $\Gamma X$. For each $j\in\{1,\ldots,n\}$,
\begin{align*}
z_j=\gamma_jg_jg\tau_x,
\end{align*}
for some $g\in\GL_2(\RR)^+$ and some $\gamma_j\in\SL_2(\ZZ)$. Therefore, let
\begin{align*}
\Lambda:=\bigcap_{j=1}^ng_j^{-1}\SL_2(\ZZ)g_j
\end{align*}
and let $\mathcal{C}$ denote a set of representatives in $\SL_2(\ZZ)$ for $\Lambda\backslash\SL_2(\ZZ)$. Note that, if we define
\begin{align*}
m_j:=\det(g_j),
\end{align*}
then, for any multiple $N$ of $m_j$, the principal congruence subgroup $\Gamma(N)$ is contained in $g^{-1}_j\SL_2(\ZZ)g_j$. In particular, if we define $N$ to be the lowest common multiple of the $m_j$, then $\Gamma(N)$ is contained in $\Lambda$. It follows that any subset of $\SL_2(\ZZ)$ mapping bijectively to $\SL_2(\ZZ/N\ZZ)$ contains a set $\mathcal{C}$, as above. Via the procedure outlined in \cite{DS:modularforms}, Exercise 1.2.2, it is straightforward to verify that we can (and do) choose $\mathcal{C}$ such that, for any $c\in\mathcal{C}$,
\begin{align*}
H(c)\leq 7N^5
\end{align*}
(though we certainly do not claim that this is the best possible bound; any polynomial bound would suffice for our purposes).

The union 
\begin{align*}
\bigcup_{c\in\mathcal{C}}c\mathcal{F}_{\HH}
\end{align*}
constitutes a fundamental set in $\HH$ for the action of $\Lambda$. Hence, there exists $c\in\mathcal{C}$ and $\lambda_g\in\Lambda$ such that
\begin{align*}
c^{-1}\lambda_gg\tau_x\in\mathcal{F}_{\HH}. 
\end{align*}
Furthermore, for each $j\in\{1,\ldots,n\}$, we can write $\lambda_g=g^{-1}_j\lambda_jg_j$, for some $\lambda_j\in\SL_2(\ZZ)$ and, hence,
\begin{align}\label{tousemartin}
z_j=\gamma_j g_j g \tau_x = \gamma_j g_j \lambda^{-1}_g c\cdot c^{-1}\lambda_g g \tau_x=\gamma_j\lambda^{-1}_jg_jc\cdot c^{-1}\lambda_g g \tau_x.
\end{align}
Therefore, by \cite{orr:siegel}, Theorem 1.2, we have ${\rm H}_1(\gamma_j\lambda^{-1}_jg_jc)\leq c_1m_j^2$, for all $j\in\{1,\ldots,n\}$.

We write
\begin{align*}
{\rm H}_1(\gamma_j\lambda^{-1}_j)={\rm H}_1(\gamma_j\lambda^{-1}_jg_jc\cdot c^{-1}g^{-1}_j)\leq c_2{\rm H}_1(\gamma_j\lambda^{-1}_jg_jc)^{\delta}{\rm H}_1(c^{-1})^{\delta}{\rm H}_1(g^{-1}_j)^{\delta},
\end{align*}
where $c_2$ and $\delta$ are positive constants not depending on $Z$, and we obtain
\begin{align*}
{\rm H}_1(\gamma_j\lambda^{-1}_j)\leq c_3N^{10\delta},
\end{align*}
for each $j\in\{1,\ldots,n\}$, where $c_3$ is a positive constant not depending on $Z$.

Conversely, by \cite{daw:aomod}, \S2,
\begin{align*}
\deg(Z)\geq [\SL_2(\ZZ):\Lambda]
\end{align*}
and, by writing the $g_j$ in Smith normal form i.e.
\begin{align*}
g_j=\gamma_{j,1}\left( \begin{array}{cc}
1 & 0 \\
0 & m_j \end{array} \right)\gamma_{j,2},
\end{align*}
where $\gamma_{j,1},\gamma_{j,2}\in\SL_2(\ZZ)$, we conclude that 
\begin{align*}
\SL_2(\ZZ)\cap g^{-1}_j\SL_2(\ZZ)g_j=\gamma^{-1}_{j,2}\Gamma_0(m_j)\gamma_{j,2},
\end{align*}
whose index in $\SL_2(\ZZ)$ is the same as $\Gamma_0(m_j)$, which is $m_j\varphi(m_j)$. It follows that
\begin{align*}
[\SL_2(\ZZ):\Lambda]\geq N,
\end{align*}
and so
\begin{align*}
{\rm H}_1(\gamma_j\lambda^{-1}_j)\leq c_3\deg(Z)^{10\delta}.
\end{align*}
Therefore, we let $\gamma:=(\gamma_j\lambda^{-1}_j)^n_{j=1}\in\Gamma$. By (\ref{tousemartin}), $z\in\gamma X$, and the result follows.

\end{proof}

Finally, we will verify Conjecture \ref{degcom} in this case.

\begin{proof}[Proof of Conjecture \ref{degcom} for $\CC^n$]

Let $Z$ be a special subvariety of $\CC^n$. Then $Z$ is a product of special subvarieties. Since there are only finitely many partitions of $\{1,\ldots,n\}$, we may assume that the product contains only one factor. If $Z$ is a special point, $H^\der$ is trivial. Therefore, we assume that $Z$ is equal to the image of $\HH$ in $\CC^n$ under the map sending $\tau\in\HH$ to the image of $(g_j\cdot\tau)^n_{j=1}$ in $\CC^n$ for elements $g_j\in\GL_2(\QQ)^+$. Then $H^\der$ is the image of $\SL_2$ under the morphism
\begin{align*}
\SL_2\rightarrow\GL_2^n:g\mapsto(g_jgg^{-1}_j)_{j=1}^n.
\end{align*}
We see from the calculations in the previous proof that the bound $\deg(Z)\leq b$ implies that the $g_j$ come from the union of finitely many double cosets $\Gamma g\Gamma$ for $g\in\GL_2(\QQ)^+$. Since each such double coset is equal to a finite union of single cosets $\Gamma h$ for $h\in\GL_2(\QQ)^+$, the result follows.

\end{proof}

\section{Main results (part 2): Conditional solutions to the counting problems}

We conclude by demonstrating how our arithmetic conjectures might be used to resolve the counting problems stated in Section \ref{reductions}. In our applications of the counting theorem, we will need the following.

\begin{lemma}\label{curve}
Let $\beta:[0,1]\rightarrow G(\RR)\times X$ be semi-algebraic. Then $\Im(\beta)$ is contained in a complex algebraic subset $B$ of $G(\CC)\times\CC^N$ of dimension at most $1$.
\end{lemma}

\begin{proof}
Let $Y$ denote the real Zariski closure of $\Im(\beta)$ in $G(\RR)\times\RR^{2N}$. In particular, $\dim Y\leq 1$. Without loss of generality, we can and do assume that $Y$ is irreducible. If $Y$ is a point then there is nothing to prove. Therefore, we can and do assume that $Y$ is an irreducible real algebraic curve. In particular, the complexification $Y_\CC$ of $Y$ in $G(\CC)\times\CC^{2N}$ is an irreducible complex algebraic curve.

Let $g_1,\ldots,g_{n^2},x_1,y_1,\ldots,x_N,y_N$ denote the real coordinate functions on $G(\RR)\times X$ and let $z_j=x_j+iy_j$ denote the coordinate functions on $\CC^N=\RR^{2N}$. If all of the coordinates functions on $\RR^{2N}$ are constant on $Y$, the result is obvious. Therefore, without loss of generality, we can and do assume that $x_1$ is not constant on $Y$. 

We claim that each of the coordinate functions $x_2,y_2,\ldots,x_N,y_N$ on $\CC^{2N}$ is algebraic over the field $\CC(z_1)$, considered as a field of functions on $Y_\CC$. To see this, note that $z_1$ is non-constant on $Y_\CC$, and so $\CC(z_1)$ has transcendence degree at least $1$. On the other hand, $\CC(z_1)$ is contained in $\CC(x_1,y_1)$, which is algebraic over $\CC(x_1)$. 

In particular, each of the functions $x_2+iy_2,\ldots,x_N+iy_N$ is algebraic over the field $\CC(z_1)$. It follows that, for each $j\geq 2$, there exists a polynomial $f_j(z_1,z_j)\in\CC[z_1,z_j]$, non-trivial in $z_j$, such that $f_j(z_1,z_j)=0$ on $Y$. Similarly, for each $k=1,\ldots, n^2$, there exists a polynomial $f_j(z_1,g_k)\in\CC[z_1,g_k]$, non-trivial in $g_k$, such that $f_k(z_1,g_k)=0$ on $Y$. In particular, $Y$ is contained in the vanishing locus of the $f_j$ and the $f_k$, which define a complex algebraic curve in $G(\CC)\times\CC^N$.
\end{proof}

We denote by $X^\vee$ the compact dual of $X$, which is a complex algebraic variety on which $G(\CC)$ acts via an algebraic morphism
\begin{align*}
G(\CC)\times X^\vee\rightarrow X^\vee.
\end{align*}
Furthermore, $X$ naturally embeds into $X^\vee$ and the embedding factors through an embedding of $\CC^N$ i.e. the Harish-Chandra realization, into $X^\vee$. We could have defined subvarieties of $X$ using $X^\vee$ in the place of $\CC^N$ but, as mentioned previously, the two notions coincide. If we have a decomposition $G^\ad=G_1\times G_2$, and thus $X=X_1\times X_2$, we have a natural decomposition
\begin{align*}
X^\vee=X^\vee_1\times X^\vee_2.
\end{align*}
Furthermore, if $(H,\mathfrak{X}_H)$ denotes a Shimura subdatum of $(G,\mathfrak{X})$ and $X_H$ is a connected component of $\mathfrak{X}_H$ contained in $X$, then $X^\vee_H$ is naturally contained in $X^\vee$. We refer the reader to \cite{ey:subvarieties}, Section 3 for more details.

\begin{theorem}\label{Opt0}
Assume that Conjecture \ref{LGO} holds and assume that either
\begin{itemize}
\item Conjecture \ref{conj} holds or
\item Conjectures \ref{count complexity}, \ref{fieldofdef}, and \ref{orr2} hold.
\end{itemize}
Then Conjecture \ref{zerodim} is true for curves i.e. if $V$ is a curve contained in $S$, then the set $\Opt_0(V)$ is finite.
\end{theorem}

\begin{proof}
We will assume that Conjecture \ref{conj} holds. The proof in the case that Conjectures \ref{count complexity}, \ref{fieldofdef}, and \ref{orr2} hold is very similar, hence we omit it. To elucidate the use of Conjectures \ref{count complexity}, \ref{fieldofdef}, and \ref{orr2} we will use them in the proof of Theorem \ref{oa0}, at the expense of making the proof longer. We suffer no loss of generality if we assume, as we will, that $V$ is Hodge generic.

Let $\Omega$ denote a finite set of semisimple subgroups of $G$ defined over $\QQ$ as in Section \ref{fah} and let $d$, $c_{\mathcal{F}}$, and $\delta_{\mathcal{F}}$ be the constants afforded to us by Conjecture \ref{conj}. Let $L$ be a finitely generated extension of $F$ contained in $\CC$ over which $V$ is defined and let $c_G$ and $\delta_G$ be the constants afforded to us by Conjecture \ref{LGO}. Let $\kappa:=2\delta_G/3\delta_\mathcal{F}$. 

We claim that there exists a positive constant $c$ such that, for any $P\in\Opt_0(V)$, we have
\begin{align*}
\#{\Aut}(\CC/L)\cdot P\leq cc_{\mathcal{F}}^{\kappa}\Delta(\langle P\rangle)^{\kappa\delta_\mathcal{F}}.
\end{align*}
This would be sufficient to prove Theorem \ref{Opt0} since then, by Conjecture \ref{LGO}, we obtain
\begin{align*}
c_G\Delta(\langle P\rangle)^{\delta_G}\leq\#{\Aut}(\CC/L)\cdot P\leq cc_{\mathcal{F}}^{\kappa}\Delta(\langle P\rangle)^{\kappa\delta_\mathcal{F}}
\end{align*}
and, rearranging this expression, we obtain
\begin{align*}
\Delta(\langle P\rangle)\leq (cc_\mathcal{F}^\kappa c^{-1}_G)^{3/\delta_G},
\end{align*}
which is a bound independent of $P$. We remind the reader that $P$ is one of only finitely many irreducible components of $\langle P\rangle\cap V$. Hence, Theorem \ref{Opt0} would follow from Lemma \ref{conjimplic} and, therefore, it remains only to prove the claim. 

To that end, for each $\sigma\in{\rm Gal}(\CC/L)$, let $z_{\sigma}\in\mathcal{V}$ be a point in $\pi^{-1}(\sigma(P))$. Therefore, by Conjecture \ref{conj}, the smallest pre-special subvariety of $X$ containing $z_\sigma$ can be written $g_{\sigma}F_\sigma(\RR)^+g^{-1}_\sigma x_\sigma$, where $F_\sigma\in\Omega$, and $g_\sigma\in G(\RR)$ and $x_\sigma\in X$ satisfy
\begin{align*}
{\rm H}_d(g_\sigma,x_\sigma)\leq c_\mathcal{F}\Delta(\langle \sigma(P)\rangle)^{\delta_\mathcal{F}}=c_\mathcal{F}\Delta(\langle P\rangle)^{\delta_\mathcal{F}}.
\end{align*}
Without loss of generality, we can and do assume that $F:=F_\sigma$ is fixed. Therefore, for each $\sigma\in{\rm Gal}(\CC/L)$, the tuple $(g_\sigma,x_\sigma,z_\sigma)$ belongs to the definable set $D$ of tuples
\begin{align*}
(g,x,z)\in G(\RR)\times X\times X\subseteq\RR^{n^2+2N}\times\RR^{2N},
\end{align*}
such that $z\in\mathcal{V}\cap gF(\RR)^+g^{-1}x$ and $x(\SSS)\subseteq gG_Fg^{-1}$. We consider $D$ as a family over a point in an omitted parameter space and choose for $c$ the constant $c(D,d,\kappa)$ afforded to us by Theorem \ref{pilawilkie} applied to $D$. Since $\Omega$ is finite, we can and do assume that $c$ does not depend on $F$. We let $\Sigma$ denote the union over ${\rm Aut}(\CC/L)$ of the tuples $(g_\sigma,x_\sigma,z_\sigma)\in D$. In particular, $\Sigma$ is contained in the subset
\begin{align*}
D(d,c_{\mathcal{F}}\Delta(\langle P\rangle)^{\delta_{\mathcal{F}}}).
\end{align*}

Let $\pi_1$ and $\pi_2$ be the projection maps from $\RR^{n^2+2N}\times\RR^{2N}$ to $\RR^{n^2+2N}$ and $\RR^{2N}$, respectively, and suppose, for the sake of obtaining a contradiction, that 
\begin{align*}
\#{\Aut}(\CC/L)\cdot P=\#\pi_2(\Sigma)>cc_{\mathcal{F}}^{\kappa}\Delta(\langle P\rangle)^{\kappa\delta_\mathcal{F}}.
\end{align*}
Then, by Theorem \ref{pilawilkie}, there exists a continuous definable function 
\begin{align*}
\beta:[0,1]\rightarrow D,
\end{align*}
such that $\beta_1:=\pi_1\circ \beta$ is semi-algebraic, $\beta_2:=\pi_2\circ \beta$ is non-constant, $\beta(0)\in\Sigma$, and $\beta_{|(0,1)}$ is real analytic. Let $z_0:=\beta_2(0)$ and let $P_0:=\pi(z_0)$. To obtain a contradiction, we will closely imitate arguments found in \cite{orr:unlikely}.

It follows from the Global Decomposition Theorem (see \cite{gr:sheaves}, p172) that there exists $0<t\leq 1$ such that $\beta_2([0,t))$ intersects only finitely many of the irreducible analytic components of $\pi^{-1}(V)$. In fact, since $\beta_{2|(0,t)}$ is real analytic, $\beta_2((0,t))$ must be wholly contained in one such component $V_1$. Since $V_1$ is closed, we conclude from the fact that $\beta$ is continuous that $V_1$ contains $\beta_2([0,t])$.  

By \cite{uy:algebraic-flows}, Theorem 1.3 (the inverse Ax-Lindemann conjecture), $\langle V_1\rangle_{\rm Zar}$ is pre-weakly special and so, since $V$ is Hodge generic in $S$, we can decompose $G^\ad=G_1\times G_2$, and thus $X=X_1\times X_2$, so that
\begin{align*}
\langle V_1\rangle_{\rm Zar}=X_1\times\{x_2\},
\end{align*}
where $x_2\in X_2$ is Hodge generic. By abuse of notation, we denote by $\pi_2$ both the projection from $G$ to $G_2$ and from $X^\vee$ to $X^\vee_2$.

Note that, for any $(g,x)\in\Im(\beta_1)$, we have $(g^{-1}x)(\SSS)\subseteq G_{F,\RR}$. If we write $G'_F$ for the largest normal subgroup of $G_F$ of non-compact type, then the properties of Shimura data imply that $g^{-1}x$ factors through $G'_{F,\RR}$ and, if we write $\mathfrak{X}'$ for the $G'_F(\RR)$ conjugacy class of $g^{-1}x$ in $\mathfrak{X}$, then, by \cite{ullmo:equidistribution}, Lemme 3.3, $(G'_F,\mathfrak{X}')$ is a Shimura subdatum of $(G,\mathfrak{X})$. Furthermore, by \cite{uy:andre-oort}, Lemma 3.7, the number of Shimura subdata $(G'_F,\mathfrak{Y})$ of $(G,\mathfrak{X})$ is finite and, by \cite{milne:intro}, Corollary 5.3, the number of connected components $Y$ of $\mathfrak{Y}$ is also finite. It follows that, after possibly replacing $t$, we can and do assume that $g^{-1}x$ belongs to one such component $Y$, which we write as $Y_1\times Y_2$, such that $F(\RR)^+$ acts transitively on $Y_1$. In particular, 
\begin{align*}
\dim Y_1=\dim\langle P_0\rangle.
\end{align*}
We let $p_2$ denote the projection from $Y^\vee$ to $Y^\vee_2$.

Let $B$ denote the complex algebraic subset of $G(\CC)\times X^{\vee}$ of dimension at most $1$ containing $\Im(\beta_1)$ afforded to us by Lemma \ref{curve}. For any $(g,x)\in B$, we have $g^{-1}x\in Y^{\vee}$. 

Let $\overline{V}_1$ denote the Zariski closure of $V_1$ in $X^\vee$ and consider the complex algebraic set
\begin{align*}
W_B:=\{(g,x,y)\in B\times Y^{\vee}:p_2(y)=p_2(g^{-1}x),\ gy\in\overline{V}_1\}.
\end{align*}
Let $V_B$ denote the Zariski closure in $X^\vee$ of the set 
\begin{align*}
\{gy:(g,x,y)\in W_B\}.
\end{align*}
Since the latter is the image of $W_B$ under an algebraic morphism, we have $\dim V_B\leq\dim W_B$.

Since $V_1$ is an irreducible complex analytic curve having uncountable intersection with $V_B$, it follows that $V_1$ is contained in $V_B$.  Therefore, $\langle V_1\rangle_{\rm Zar}$ is contained in $V_B$ also, and so
\begin{align}\label{dimX1}
\dim X_1=\dim\langle V_1\rangle_{\rm Zar}\leq\dim V_B\leq \dim W_B.
\end{align}

Now, for each $(g,x)\in B$, consider the fibre $W_{(g,x)}$ of $W_B$ over $(g,x)$ i.e. the set
\begin{align*}
\{y\in Y^\vee:p_2(y)=p_2(g^{-1}x),\ \pi_2(y)=\pi_2(g)^{-1}x_2\}.
\end{align*}
Since $P_0\in V$, it follows that $\pi_2(F)=G_2$ and so, for any $y\in Y_2^\vee$, the natural projection
\begin{align*}
Y^\vee_1\times\{y\}\rightarrow X^\vee_2
\end{align*}
is an equivariant morphism of $F(\CC)$-homogeneous spaces. In particular, its fibres are equidimensional of dimension
\begin{align*}
\dim Y^\vee_1-\dim X^\vee_2=\dim Y_1-\dim X_2.
\end{align*}
Since $W_{(g,x)}$ is contained in such a fibre, we have
\begin{align*}
\dim W_{(g,x)}\leq \dim Y_1-\dim X_2\leq\dim X-2-\dim X_2=\dim X_1-2,
\end{align*}
where we use the fact that $P_0\in\Opt_0(V)$, hence,
\begin{align*}
\dim Y_1=\delta(P_0)\leq\delta(V)-1=\dim X-2.
\end{align*}
Since this holds for all $(g,x)\in B$ and $\dim B\leq 1$, we conclude that
\begin{align*}
\dim W_B\leq\dim X_1-1,
\end{align*}
which contradicts (\ref{dimX1}).
\end{proof}

Of course, Theorem \ref{Opt0} is not really satisfactory in the sense that it only deals with curves. One would hope that, for $V$ of arbitrary dimension, a path such as $\beta$ would yield, via the weak hyperbolic Ax-Schanuel conjecture, a positive dimensional subvariety of $V$, containing a conjugate of $P$, having defect at most $\delta(P)$, thus contradicting the optimality of $P$. However, the authors haven't been able to carry out this procedure. Instead, the very same idea appears to work when one attempts to contradict the membership of a point in the open-anomalous locus. The difference is that we are only required to bound the weakly special defect, as opposed to the defect itself.

\begin{theorem}\label{oa0}
Assume that Conjecture \ref{LGOoa} holds and assume that the weak hyperbolic Ax-Schanuel conjecture is true. Assume also that, either
\begin{itemize}
\item Conjecture \ref{conj} holds, or
\item Conjectures \ref{count complexity}, \ref{fieldofdef}, and \ref{orr2} hold.
\end{itemize}
Then, Conjecture \ref{oafinite} is true i.e. if $V$ is a subvariety of $S$, then the set 
\begin{align*}
V^{\rm oa}\cap S^{[1+\dim V]}
\end{align*}
is finite.
\end{theorem}

\begin{proof}

We will assume that Conjectures \ref{count complexity}, \ref{fieldofdef}, and \ref{orr2} hold. The proof in the case that Conjecture \ref{conj} holds is very similar, hence we omit it. We used Conjecture \ref{conj} in the proof of Theorem \ref{Opt0}.

Let $\Omega$ denote a finite set of semisimple subgroups of $G$ defined over $\QQ$ as in Section \ref{fah}. Let $c_{\Gamma}$ and $\delta_{\Gamma}$ be the constants afforded to us by Conjecture \ref{orr2}, let $d$, $c_{\rm H}$, and $\delta_{\rm H}$ be the constants afforded to us by Lemma \ref{heightx}, and let
\begin{align*}
c:=\max\{c_{\rm H},c_{\Gamma}\}\text{ and }\delta:=\max\{\delta_{\rm H},\delta_{\Gamma}\}.
\end{align*}
Let $L'$ be a finitely generated extension of $F$ contained in $\CC$ over which $V$ is defined and let $c_G$ and $\delta_G$ be the constants afforded to us by Conjecture \ref{LGO}. Let $\kappa:=\delta_G/3$, and let $c_\kappa$ be the constant afforded to us by Conjecture \ref{fieldofdef}. Let 
\begin{align*}
P\in V^{\rm oa}\cap S^{[1+\dim V]}
\end{align*}
and let $L$ and $H$ be, respectively, the finite field extension of $F$ and the semisimple subgroup of $G$ defined over $\QQ$ of non-compact type afforded to us by Conjecture \ref{fieldofdef} applied to $\langle P\rangle$. Replacing $L$ by its compositum with $L'$, we have
\begin{align*}
[L:L']\leq c_\kappa\Delta(\langle P\rangle)^\kappa.
\end{align*}

We claim that there exists a positive constant $c_3$, independent of $P$, such that
\begin{align*}
\#{\Aut}(\CC/L)\cdot P\leq c_3c^{\frac{\kappa}{\delta}}\Delta(\langle P\rangle)^{\kappa}.
\end{align*}
This would be sufficient to prove Theorem \ref{oa0} since then, by Conjecture \ref{LGOoa}, we obtain
\begin{align*}
\frac{c_G}{c_\kappa}\Delta(\langle P\rangle)^{2\kappa}\leq\frac{1}{[L:L']}\#{\Aut}(\CC/L')\cdot P=\#{\Aut}(\CC/L)\cdot P\leq  c_3c^{\frac{\kappa}{\delta}}\Delta(\langle P\rangle)^{\kappa}
\end{align*}
and, rearranging this expression, we obtain
\begin{align*}
\Delta(\langle P\rangle)\leq (c_3c^{\frac{\kappa}{\delta}}c_\kappa c^{-1}_G)^{\frac{1}{\kappa}},
\end{align*}
which is a bound independent of $P$. We remind the reader that, as explained in Remark \ref{oainopt}, $P\in\Opt_0(V)$ and, therefore, $P$ is one of only finitely many irreducible components of $\langle P\rangle\cap V$. Hence, Theorem \ref{oa0} would follow from Conjecture \ref{count complexity} and it remains only, therefore, to prove the claim.

By Conjecture \ref{fieldofdef}, for each $\sigma\in{\Aut}(\CC/L)$,
\begin{align*}
\sigma(\langle P\rangle)=\pi(H(\RR)^+x_{\sigma}),
\end{align*}
where $H(\RR)^+x_{\sigma}$ is a pre-special subvariety of $X$ intersecting $\mathcal{F}$. By Lemma \ref{heightx}, we can and do assume that
\begin{align*}
{\rm H}_d(x_\sigma)\leq c_{\rm H}\Delta(\sigma(\langle P\rangle))^{\delta_{\rm H}}=c_{\rm H}\Delta(\langle P\rangle)^{\delta_{\rm H}}.
\end{align*}
We let $z_{\sigma}\in\mathcal{V}$ be a point in $\pi^{-1}(\sigma(P))$, so that
\begin{align*}
z_{\sigma}\in\Gamma H(\RR)^+x_{\sigma}
\end{align*}
and so, by Conjecture \ref{orr2}, there exists $\gamma_{\sigma}\in\Gamma$ satisfying
\begin{align*}
{\rm H}_1(\gamma_{\sigma})\leq c_{\Gamma}\Delta(\langle P\rangle)^{\delta_{\Gamma}}
\end{align*}
such that $z_{\sigma}\in \gamma_{\sigma}H(\RR)^+x_{\sigma}$.

By definition, there exists $F\in\Omega$ and $g\in G(\RR)$ such that $H_{\RR}$ is equal to $gF_{\RR}g^{-1}$. In particular, for each $\sigma\in{\Aut}(\CC/L)$, the tuple $(g,(\gamma_{\sigma},x_{\sigma}),z_{\sigma})$ belongs to the definable family $D$ of tuples
\begin{align*}
(g,(\gamma,x),z)\in G(\RR)\times [G(\RR)\times X]\times X\subseteq\RR^{n^2}\times\RR^{n^2+2N}\times\RR^{2N},
\end{align*}
parametrised by $G(\RR)$, such that 
\begin{align*}
z\in\mathcal{V}\cap\gamma gF(\RR)^+g^{-1}x,\text{ and }x(\SSS)\subseteq gG_Fg^{-1}.
\end{align*}
We choose, then, for $c_3$ the constant $c(D,d,\kappa/\delta)$ afforded to us by Theorem \ref{pilawilkie} applied to $D$. Since, $\Omega$ is finite, we can and do assume that $c_3$ does not depend on $F$. We let $\Sigma$ denote the union over ${\Aut}(\CC/L)$ of the tuples $((\gamma_\sigma,x_{\sigma}),z_{\sigma})\in D_{g}$ (to use the notation of Section \ref{count}). In particular, $\Sigma$ is contained in the subset
\begin{align*}
D_{g}(d,c\Delta(Z)^\delta).
\end{align*}

Let $\pi_1$ and $\pi_2$ be the projection maps from $\RR^{n^2+2N}\times\RR^{2N}$ to $\RR^{n^2+2N}$ and $\RR^{2N}$, respectively, and suppose, for the sake of obtaining a contradiction, that 
\begin{align*}
\#{\Aut}(\CC/L)\cdot P=\#\pi_2(\Sigma)>c_3c^{\frac{\kappa}{\delta}}\Delta(Z)^{\kappa}.
\end{align*}
Then, by Theorem \ref{pilawilkie}, there exists a continuous definable function 
\begin{align*}
\beta:[0,1]\rightarrow D_{g},
\end{align*}
such that $\beta_1:=\pi_1\circ \beta$ is semi-algebraic, $\beta_2:=\pi_2\circ \beta$ is non-constant, $\beta(0)\in\Sigma$, and $\beta_{|(0,1)}$ is real analytic. Let $z_0:=\beta_2(0)$ and $(\gamma_0,x_0):=\beta_1(0)$. Denote by $P_0$ the point $\pi(z_0)$ and denote by $X_0$ the pre-special subvariety $\gamma_0H(\RR)^+x_0$. 

We claim that there exists a positive dimensional intersection component $A$ of $\pi^{-1}(V)$ containing $z_0$. To see this, let $W$ denote the union of the totally geodesic subvarieties $\gamma H(\RR)^+x$ of $X$, where $(\gamma,x)$ varies over $\Im(\beta_1)$, and let $\overline{W}$ denote the Zariski closure of $W$ in $X^\vee$. The irreducible analytic components of $\overline{W}\cap\pi^{-1}(V)$ are, by definition, intersection components of $\pi^{-1}(V)$. It follows from the Global Decomposition Theorem (see \cite{gr:sheaves}, p172) that there exists $0<t\leq 1$ such that $\beta_2([0,t))$ intersects only finitely many of said components. In fact, since $\beta_{2|(0,t)}$ is real analytic, $\beta_2((0,t))$ must be wholly contained in one such component $A$. Since $A$ is closed, we conclude from the fact that $\beta$ is continuous that $A$ contains $\beta_2([0,t])$, which proves the claim.

Let $B$ denote a Zariski optimal intersection component of $\pi^{-1}(V)$ containing $A$ such that
\begin{align*}
\delta_{\rm Zar}(B)\leq\delta_{\rm Zar}(A),
\end{align*}
and let $Z$ denote the Zariski closure of $\pi(B)$ in $S$. By the weak hyperbolic Ax-Schanuel conjecture, $\langle B\rangle_{\rm Zar}$ is pre-weakly special and, as in the proof of Lemma \ref{upstairs}, 
\begin{align*}
\langle Z\rangle_{\rm ws}=\pi(\langle B\rangle_{\rm Zar}).
\end{align*}
Therefore, we have $\dim Z\geq 1$ and, also,
\begin{align*}
\dim Z\geq\dim B\geq\dim\langle B\rangle_{\rm Zar}-\delta_{\rm Zar}(A)\geq\dim\langle Z\rangle_{\rm ws}-(\dim\overline{W}-1),
\end{align*}
where we use the fact that $\delta_{\rm Zar}(A)$ is at most $\dim\overline{W}-1$. We claim that $\dim\overline{W}-1\leq\dim X_0$, which would conclude the proof as
\begin{align*}
\dim X_0\leq\dim S-\dim V-1
\end{align*}
and this would imply that $Z\in {\rm an}(V)$, which is not allowed as $P_0\in Z$.

Therefore, it remains to prove the claim. However, this is easy to prove working with complex duals and using the methods explained in the proof of Theorem \ref{Opt0}.

\end{proof}

\section{A brief note on special anomalous subvarieties}

In their paper \cite{bombieri2007anomalous}, Bombieri, Masser, and Zannier also defined what they referred to as a torsion anomalous subvariety. We will make the analogous definition in the context of Shimura varieties. Let $V$ be a subvariety of $S$.

\begin{definition}
A subvariety $W$ of $V$ is called {\bf special anomalous} in $V$ if
\begin{align*}
\dim W\geq \max\{1,1+\dim V+\dim \langle W\rangle-\dim S\}.
\end{align*}
 A subvariety of $V$ is {\bf maximal special anomalous} in $V$ if it is special anomalous in $V$ and not strictly contained in another subvariety of $V$ that is also special anomalous in $V$.
\end{definition}

The similarity with the definition of an atypical subvariety is clear. Indeed, it is immediate that a positive dimensional subvariety of $V$ that is atypical with respect to $V$ is special anomalous in $V$. However, since a subvariety $W$ of $V$ that is special anomalous in $V$ is not necessarily an irreducible component of $V\cap\langle W\rangle$, it is not necessarily the case that $W$ is atypical with respect to $V$. Nonetheless, it follows that the properties of being maximal special anomalous and atypical are equivalent for positive dimensional subvarieties of $V$. In particular, Conjecture \ref{zp'} implies the following anologue of the Torsion Openness conjecture of Bombieri, Masser, and Zannier.

\begin{conjecture}\label{TOC1}
There are only finitely many subvarieties of $V$ that are maximal special anomalous in $V$.
\end{conjecture}

Of course, one would more naturally translate the Torsion Openness conjecture as follows.

\begin{conjecture}\label{TOC2}
The complement $V^{\rm sa}$ in $V$ of the subvarieties of $V$ that are special anomalous in $V$ is open in $V$.
\end{conjecture}

However, these two formulations are equivalent. Indeed, the statement that Conjecture \ref{TOC1} implies Conjecture \ref{TOC2} is obvious. On the other hand, suppose that Conjecture \ref{TOC2} were true. Then the union of all subvarieties of $V$ that are positive dimensional and atypical with respect to $V$ would be closed in $V$. In particular, we could write it as a finite union of subvarieties of $V$. However, since there are only countably many subvarieties of $V$ that are atypical with respect to $V$, it follows that each member of the aforementioned union would be atypical with respect to $V$.

Of course, if one could prove Conjecture \ref{TOC1}, one would reduce the Zilber-Pink conjecture to the following analogue of the Torsion Finiteness conjecture of Bombieri, Masser, and Zannier.

\begin{conjecture}\label{TFC1}
There are only finitely many points in $V$ that are maximal atypical with respect to $V$.
\end{conjecture}

In this article, we have concerned ourselves with optimal subvarieties. Now, it is straightforward to verify that a subvariety $W$ of $V$ that is optimal in $V$ is atypical with respect to $V$. However, it is not necessarily the case that $W$ is maximal atypical. On the other hand, a subvariety of $V$ that is maximal atypical with respect to $V$ is optimal. In particular, the points in $V$ that are maximal atypical with respect to $V$ constitute a (possibly proper) subset of $\Opt_0(V)$. 

Again, it would be more natural to translate the Torsion Finiteness conjecture as follows.

\begin{conjecture}\label{TFC2}
For any integer $d$, let $S^{[d]}$ denote the union of the special subvarieties contained in $S$ having codimension at least $d$. Then
\begin{align*}
V^{\rm sa}\cap S^{[1+\dim V]}
\end{align*}
is finite. Equivalently, there are only finitely many points $P\in V^{\rm sa}$ such that
\begin{align*}
\dim\langle P\rangle\leq\dim S-\dim V-1.
\end{align*}
\end{conjecture}

However, the two formulations are also equivalent. Indeed, Conjecture \ref{TFC2} implies Conjecture \ref{TFC1} because a point $P\in V$ that is maximal atypical with respect to $V$ is contained in $V^{\rm sa}$ and
\begin{align*}
\dim\langle P\rangle\leq\dim S-\dim V-1.
\end{align*}
On the other hand, suppose that Conjecture \ref{TFC1} were true and consider a point $P\in V^{\rm sa}$ such that
\begin{align*}
\dim\langle P\rangle\leq\dim S-\dim V-1.
\end{align*}
Then $P$ is a component of $\langle P\rangle\cap V$. Otherwise, such a component $W$ containing $P$ would be special anomalous in $V$, which would contradict the fact that $P\in V^{\rm sa}$. Therefore, $P$ is atypical with respect to $V$ and, in fact, maximal atypical with respect to $V$.

In their article \cite{bombieri2007anomalous}, Bombieri, Masser, and Zannier showed that, in fact, the Torsion Openness conjecture implies the Torsion Finiteness conjecture. We imitate their argument to show the following.

\begin{proposition}\label{TOimpliesTF}
Let $Y(1)$ denote the modular curve associated with $\SL_2(\ZZ)$. If Conjecture \ref{TOC2} is true for $S\times Y(1)$, then Conjecture \ref{TFC2} is true for $S$. 
\end{proposition}

\begin{proof} We denote by $U$ the complement in $V^{\mathrm{sa}}$ of $V^{\mathrm{sa}}\cap S^{[\dim V+1]}$. Regarding $V\times Y(1)$ as a subvariety of the Shimura variety $S\times Y(1)$, we claim that
\begin{align}\label{product}
(V\times Y(1))^{\mathrm{sa}}=U\times Y(1).
  \end{align}

To see this, first let $\hat{Y}$ be a special anomalous subvariety of $V\times Y(1)$. Then
\begin{align*}
\dim \hat{Y} &\geq 1+\dim \langle \hat{Y}\rangle+\dim (V\times Y(1))-\dim (S\times Y(1)) \\
&= 1+\dim \langle \hat{Y}\rangle+(1+\dim V)-(1+\dim S)\\
&= 1+\dim \langle \hat{Y}\rangle+\dim V-\dim S.
\end{align*}

We denote by $\pi_S$ the projection $\pi_S: S\times Y(1)\rightarrow S$. Then the Zariski closure $Y$ of $\pi_S(\hat{Y})$ lies in the special subvariety  $\pi_S(\langle \hat{Y}\rangle)$ of $S$. If $\dim Y=\dim \hat{Y}$, then 
\begin{align*}
\dim Y &=\dim \hat{Y}\\
&\geq 1+\dim \langle \hat{Y}\rangle+\dim V-\dim S \\
&\geq 1+\dim \pi_S(\langle \hat{Y}\rangle)+\dim V-\dim S\\
&\geq 1+\dim \langle Y\rangle+\dim V-\dim S
\end{align*}
and $Y$ is special anomalous in $V$. If $1\leq \dim Y< \dim \hat{Y}$, then $\dim Y=\dim \hat{Y}-1$ and $\hat{Y}=Y\times Y(1)$, which implies $\langle \hat{Y}\rangle=\langle Y\rangle\times Y(1)$. Therefore,
\begin{align*}
\dim Y &=\dim \hat{Y}-1\\
&\geq 1+\dim \langle \hat{Y}\rangle+\dim V-\dim S-1 \\
&= 1+\dim \langle Y\rangle +1+\dim V-\dim S-1\\
&= 1+\dim \langle Y\rangle+\dim V-\dim S
\end{align*}
and $Y$ is special anomalous in $V$. In both cases, $\hat{Y}$ is contained in 
\begin{align*}
Y\times Y(1)\subseteq (V\backslash V^{\mathrm{sa}})\times Y(1)\subseteq (V\backslash U)\times Y(1)=(V\times Y(1))\backslash (U\times Y(1)).
\end{align*}

Finally, if $\dim Y=0$, then $\hat{Y}=\{P\}\times Y(1)$ for some $P\in S$. Since $\hat{Y}$ is special anomalous, we have 
\begin{align*}
1 &=\dim \hat{Y}\\
&\geq 1+\dim \langle \hat{Y}\rangle+(1+\dim V)-(1+\dim S)\\
&= 1+\dim \langle P\rangle+1 +\dim V-\dim S
\end{align*}
and $P\in S^{[1+\dim V]}$. Therefore 
\begin{align*}
\hat{Y} &\subseteq V\backslash  (V-S^{[1+\dim V]})\times Y(1)\\
&\subseteq V\backslash  (V^{\mathrm{sa}}-S^{[1+\dim V]})\times Y(1) \\
&=(V\backslash U)\times Y(1)\\
&= (V\times Y(1))\backslash (U\times Y(1)).
\end{align*}
We conclude that $(V\times Y(1))^{\rm sa}\subseteq(U\times Y(1))$.

On the other hand, for any $P\in V\backslash U=V\backslash (V^{\mathrm{sa}}-S^{[1+\dim V]})$, we have either $P\notin V^{\mathrm{sa}}$ or $P\in V^{\rm sa}\cap S^{[1+\dim V]}$.

If $P\notin V^{\mathrm{sa}}$, then $P$ is contained in a special anomalous $Y$ of $V$. Therefore, 
\begin{align*}
\dim (Y\times Y(1)) &=1+\dim Y\\
&\geq 1+1+\dim \langle Y\rangle +\dim V-\dim S \\
&=1+\dim \langle Y\times Y(1)\rangle+ \dim (V\times Y(1))-\dim (S\times Y(1))
\end{align*}
and $Y\times Y(1)\subset (V\times Y(1))\backslash (U\times Y(1))$ is special anomalous.

If $P\in V^{\mathrm{sa}}\cap S^{[1+\dim V]}$, we let $\hat{Y}=\{P\}\times Y(1)$. Then 
\begin{align*}
\dim \hat{Y}&=1\\
&\geq 1+1+\dim \langle P\rangle +\dim V-\dim S \\
&=1+\dim \langle \hat{Y}\rangle+ \dim (V\times Y(1))-\dim (S\times Y(1))
\end{align*}
and $\hat{Y}\subseteq (V\times Y(1))\backslash (U\times Y(1))$ is special anomalous. We conclude that $(U\times Y(1))\subseteq(V\times Y(1))^{\rm sa}$, which proves (\ref{product}).

Therefore, if Conjecture \ref{TOC2} is true for $S\times Y(1)$, we conclude that $U\times Y(1)$ is open in $V\times Y(1)$. Therefore, $U$ is open in $V$. On the other hand, $V^{\mathrm{sa}}$ is also open in $V$, whereas $V^{\rm sa}\cap S^{[1+\dim V]}$ is at most countable since there are only countably many special subvarieties and each point in $V^{\rm sa}\cap S^{[1+\dim V]}$ is an irreducible component of $V\cap Z$ for some special subvariety $Z$ of codimension at most $1+\dim V$. It follows that $V^{\rm sa}\cap S^{[1+\dim V]}$ is finite.
\end{proof}

\newpage

\bibliographystyle{abbrv}
\bibliography{basic}

\begin{thebibliography}{10}

\bibitem{AGHM:colmez}
F.~Andreatta, E.~Z. Goren, B.~Howard, and K.~Madapusi~Pera.
\newblock Faltings heights of abelian varieties with complex multiplication.
\newblock {\em Ann. of Math. (2)}, 187(2):391--531, 2018.

\bibitem{Ax71}
J.~Ax.
\newblock On {S}chanuel's conjectures.
\newblock {\em Ann. of Math. (2)}, 93:252--268, 1971.

\bibitem{Ax72}
J.~Ax.
\newblock Some topics in differential algebraic geometry. {I}. {A}nalytic
  subgroups of algebraic groups.
\newblock {\em Amer. J. Math.}, 94:1195--1204, 1972.

\bibitem{bb:compactification}
W.~L. Baily, Jr. and A.~Borel.
\newblock Compactification of arithmetic quotients of bounded symmetric
  domains.
\newblock {\em Ann. of Math. (2)}, 84:442--528, 1966.

\bibitem{bombieri2007anomalous}
E.~Bombieri, D.~Masser, and U.~Zannier.
\newblock Anomalous subvarieties---structure theorems and applications.
\newblock {\em Int. Math. Res. Not. IMRN}, (19):Art. ID rnm057, 33, 2007.

\bibitem{BDR}
M.~Borovoi, C.~Daw, and J.~Ren.
\newblock Conjugation of semisimple subgroups over real number fields of
  bounded degree.
\newblock Available at https://arxiv.org/abs/1802.05894.

\bibitem{CMPZ:torus}
L.~Capuano, D.~Masser, J.~Pila, and U.~Zannier.
\newblock Rational points on {G}rassmannians and unlikely intersections in
  tori.
\newblock {\em Bull. Lond. Math. Soc.}, 48(1):141--154, 2016.

\bibitem{daw:aomod}
C.~Daw.
\newblock A simplified proof of the {A}ndr\'e-{O}ort conjecture for products of
  modular curves.
\newblock {\em Arch. Math. (Basel)}, 98(5):433--440, 2012.

\bibitem{daw:book}
C.~Daw.
\newblock The {A}ndr\'e-{O}ort conjecture via o-minimality.
\newblock In {\em O-minimality and diophantine geometry}, volume 421 of {\em
  London Math. Soc. Lecture Note Ser.}, pages 129--158. Cambridge Univ. Press,
  Cambridge, 2015.

\bibitem{Daw2016}
C.~Daw and M.~Orr.
\newblock Heights of pre-special points of {S}himura varieties.
\newblock {\em Mathematische Annalen}, 365(3):1305--1357, 2016.

\bibitem{DS:modularforms}
F.~Diamond and J.~Shurman.
\newblock {\em A first course in modular forms}, volume 228 of {\em Graduate
  Texts in Mathematics}.
\newblock Springer-Verlag, New York, 2005.

\bibitem{edix:mod}
B.~Edixhoven.
\newblock Special points on products of modular curves.
\newblock {\em Duke Math. J.}, 126(2):325--348, 2005.

\bibitem{EY:subvar}
B.~Edixhoven and A.~Yafaev.
\newblock Subvarieties of {S}himura varieties.
\newblock {\em Ann. of Math. (2)}, 157(2):621--645, 2003.

\bibitem{basopen}
S.~J. Edixhoven, B.~J.~J. Moonen, and F.~Oort.
\newblock Open problems in algebraic geometry.
\newblock {\em Bull. Sci. Math.}, 125(1):1--22, 2001.

\bibitem{gao:reduction}
Z.~Gao.
\newblock About the mixed {A}ndr\'e-{O}ort conjecture: reduction to a lower
  bound for the pure case.
\newblock {\em C. R. Math. Acad. Sci. Paris}, 354(7):659--663, 2016.

\bibitem{gao:AO}
Z.~Gao.
\newblock Towards the {A}ndre-{O}ort conjecture for mixed {S}himura varieties:
  the {A}x-{L}indemann theorem and lower bounds for {G}alois orbits of special
  points.
\newblock {\em J. Reine Angew. Math.}, 732:85--146, 2017.

\bibitem{gr:sheaves}
H.~Grauert and R.~Remmert.
\newblock {\em Coherent analytic sheaves}, volume 265 of {\em Grundlehren der
  Mathematischen Wissenschaften [Fundamental Principles of Mathematical
  Sciences]}.
\newblock Springer-Verlag, Berlin, 1984.

\bibitem{habegger:abelian}
P.~Habegger.
\newblock Intersecting subvarieties of abelian varieties with algebraic
  subgroups of complementary dimension.
\newblock {\em Invent. Math.}, 176(2):405--447, 2009.

\bibitem{hp:beyond}
P.~Habegger and J.~Pila.
\newblock Some unlikely intersections beyond {A}ndr\'e-{O}ort.
\newblock {\em Compos. Math.}, 148(1):1--27, 2012.

\bibitem{hp:o-min}
P.~Habegger and J.~Pila.
\newblock O-minimality and certain atypical intersections.
\newblock {\em Annales scientifiques de l'{\'E}cole Normale Sup{\'e}rieure},
  49(4):813--858, 2016.

\bibitem{hartshorne1977algebraic}
R.~Hartshorne.
\newblock {\em Algebraic geometry}.
\newblock Springer-Verlag, New York-Heidelberg, 1977.
\newblock Graduate Texts in Mathematics, No. 52.

\bibitem{kuy:ax-lindemann}
B.~Klingler, E.~Ullmo, and A.~Yafaev.
\newblock The hyperbolic {A}x-{L}indemann-{W}eierstrass conjecture.
\newblock {\em Publications math{\'e}matiques de l'IH{\'E}S}, 123(1):333--360,
  2016.

\bibitem{KY:AO}
B.~Klingler and A.~Yafaev.
\newblock The {A}ndr\'e-{O}ort conjecture.
\newblock {\em Ann. of Math. (2)}, 180(3):867--925, 2014.

\bibitem{maurin}
G.~Maurin.
\newblock Courbes alg\'ebriques et \'equations multiplicatives.
\newblock {\em Math. Ann.}, 341(4):789--824, 2008.

\bibitem{milne:intro}
J.~S. Milne.
\newblock Introduction to {S}himura varieties.
\newblock In {\em Harmonic analysis, the trace formula, and {S}himura
  varieties}, volume~4 of {\em Clay Math. Proc.}, pages 265--378. Amer. Math.
  Soc., Providence, RI, 2005.

\bibitem{MPT:AS}
N.~Mok, J.~Pila, and J.~Tsimerman.
\newblock Ax-{S}chanuel for {S}himura varieties.
\newblock Available at https://arxiv.org/abs/1711.02189.

\bibitem{Moonen:linear1}
B.~Moonen.
\newblock Linearity properties of {S}himura varieties. {I}.
\newblock {\em J. Algebraic Geom.}, 7(3):539--567, 1998.

\bibitem{orr:unlikely}
M.~Orr.
\newblock Unlikely intersections involving {H}ecke correspondences.
\newblock Available at https://arxiv.org/abs/1710.04092.

\bibitem{orr:siegel}
M.~Orr.
\newblock Height bounds and the {S}iegel property.
\newblock {\em Algebra Number Theory}, 12(2):455--478, 2018.

\bibitem{pt:axlindemann}
J.~Pila and J.~Tsimerman.
\newblock Ax-{L}indemann for {$\mathcal{A}_g$}.
\newblock {\em Ann. of Math. (2)}, 179(2):659--681, 2014.

\bibitem{pila2014ax}
J.~Pila and J.~Tsimerman.
\newblock Ax-{S}chanuel for the {$j$}-function.
\newblock {\em Duke Math. J.}, 165(13):2587--2605, 2016.

\bibitem{pink:published}
R.~Pink.
\newblock A combination of the conjectures of {M}ordell-{L}ang and
  {A}ndr\'e-{O}ort.
\newblock In {\em Geometric methods in algebra and number theory}, volume 235
  of {\em Progr. Math.}, pages 251--282. Birkh\"auser Boston, Boston, MA, 2005.

\bibitem{pink:generalisation}
R.~Pink.
\newblock A common generalization of the conjectures of {A}ndr{\'e}-{O}ort,
  {M}anin-{M}umford, and {M}ordell-{L}ang.
\newblock {\em Unpublished (Apr. 17th 2005), available at
  https://people.math.ethz.ch/~pink/publications.html}, 2005.

\bibitem{remond2009intersection}
G.~R{\'e}mond.
\newblock Intersection de sous-groupes et de sous-vari\'et\'es. {III}.
\newblock {\em Comment. Math. Helv.}, 84(4):835--863, 2009.

\bibitem{Tsi2015minimality}
J.~Tsimerman.
\newblock Ax-schanuel and o-minimality.
\newblock In G.~O. Jones and A.~J. Wilkie, editors, {\em O-Minimality and
  Diophantine Geometry}, pages 216--221. Cambridge University Press, 2015.
\newblock Cambridge Books Online.

\bibitem{tsimerman:AO}
J.~Tsimerman.
\newblock The {A}ndr\'e-{O}ort conjecture for {$\mathcal{A}_g$}.
\newblock {\em Ann. of Math. (2)}, 187(2):379--390, 2018.

\bibitem{ullmo:equidistribution}
E.~Ullmo.
\newblock Equidistribution de sous-vari\'et\'es sp\'eciales. {II}.
\newblock {\em J. Reine Angew. Math.}, 606:193--216, 2007.

\bibitem{ullmo:applications}
E.~Ullmo.
\newblock Applications du th\'eor\`eme d'{A}x-{L}indemann hyperbolique.
\newblock {\em Compos. Math.}, 150(2):175--190, 2014.

\bibitem{ey:subvarieties}
E.~Ullmo and A.~Yafaev.
\newblock A characterization of special subvarieties.
\newblock {\em Mathematika}, 57(2):263--273, 2011.

\bibitem{uy:andre-oort}
E.~Ullmo and A.~Yafaev.
\newblock Galois orbits and equidistribution of special subvarieties: towards
  the {A}ndr\'e-{O}ort conjecture.
\newblock {\em Ann. of Math.}, 180:823--865, 2014.

\bibitem{uy:galois}
E.~Ullmo and A.~Yafaev.
\newblock Nombre de classes des tores de multiplication complexe et bornes
  inf\'erieures pour les orbites galoisiennes de points sp\'eciaux.
\newblock {\em Bull. Soc. Math. France}, 143(1):197--228, 2015.

\bibitem{uy:algebraic-flows}
E.~Ullmo and A.~Yafaev.
\newblock Algebraic flows on {S}himura varieties.
\newblock {\em Manuscripta Math.}, 155(3-4):355--367, 2018.

\bibitem{vdD:tame-topology}
L.~van~den Dries.
\newblock {\em Tame topology and o-minimal structures}, volume 248 of {\em
  London Mathematical Society Lecture Note Series}.
\newblock Cambridge University Press, Cambridge, 1998.

\bibitem{vdDM:Ranexp}
L.~van~den Dries and C.~Miller.
\newblock On the real exponential field with restricted analytic functions.
\newblock {\em Israel J. Math.}, 85(1-3):19--56, 1994.

\bibitem{yafaevduke}
A.~Yafaev.
\newblock A conjecture of {Y}ves {A}ndr\'e's.
\newblock {\em Duke Math. J.}, 132(3):393--407, 2006.

\bibitem{YZ:colmez}
X.~Yuan and S.-W. Zhang.
\newblock On the averaged {C}olmez conjecture.
\newblock {\em Ann. of Math. (2)}, 187(2):533--638, 2018.

\bibitem{zilber:exponential}
B.~Zilber.
\newblock Exponential sums equations and the {S}chanuel conjecture.
\newblock {\em J. London Math. Soc. (2)}, 65(1):27--44, 2002.

\end{thebibliography}

\begin{flushleft}
  Christopher Daw \\
  Department of Mathematics and Statistics, \\
  University of Reading, \\
  Whiteknights, \\
  PO Box 217, \\
  Reading, \\
  Berkshire RG6 6AH, \\
  United Kingdom. \\
  E-mail address: \href{mailto:chris.daw@reading.ac.uk}{\texttt{chris.daw@reading.ac.uk}}
\end{flushleft}

\begin{flushleft}
  Jinbo Ren \\
  Institut des Hautes Études Scientifiques, \\
  Le Bois-Marie 35, route de Chartres, \\
  91440 Bures-sur-Yvette, France. \\

  E-mail address: \href{mailto:renjinbo@ihes.fr}{\texttt{renjinbo@ihes.fr}}
\end{flushleft}

\end{document}